\documentclass[11pt, twoside]{article}
\usepackage{amsfonts,amsmath,amssymb,amsthm}
\usepackage[mathscr]{euscript}
\usepackage{indentfirst}
\usepackage{enumerate}
\usepackage{color}
\usepackage{lipsum}
\usepackage{paralist}
\usepackage{graphics}
\usepackage{graphicx}
\usepackage{subfigure}
\usepackage{tikz}
\usepackage{epstopdf}
\usepackage{float}\usepackage{mathrsfs,caption}

\graphicspath{{Figures/}}
\usepackage{paralist}
\newtheorem{theo}{Theorem}[section]

\numberwithin{equation}{section}

\newcommand{\lbl}[1]{\label{#1}}
\allowdisplaybreaks[4]

\newcommand{\be}{\begin{equation}}
\newcommand{\ee}{\end{equation}}
\newcommand\bes{\begin{eqnarray}} \newcommand\ees{\end{eqnarray}}
\newcommand{\bess}{\begin{eqnarray*}}
\newcommand{\eess}{\end{eqnarray*}}
\newcommand{\bbbb}{\left\{\begin{aligned}}
\newcommand{\nnnn}{\end{aligned}\right.}
\newcommand{\bea}{\begin{align*}}
\newcommand{\eea}{\end{align*}}

\newcommand\ep{\varepsilon}
\newcommand\kk{\left}
\newcommand\rr{\right}
\newcommand\dd{\displaystyle}
\newcommand\vp{\varphi}

\newcommand\lm{\lambda}
\newcommand\nm{\nonumber}
\newcommand\yy{\infty}
\newcommand\qq{\eqref}

\newcommand\ud{\underline}

\newcommand\oo{\Omega}

\let\oldequation\equation
\let\oldendequation\endequation

\renewenvironment{equation}
{\linenomathNonumbers\oldequation}
{\oldendequation\endlinenomath}

\begin{document}\thispagestyle{empty}
\setlength{\abovedisplayskip}{7pt}
\setlength{\belowdisplayskip}{7pt}

\begin{center}{\Large\bf On a reaction-diffusion virus model with general}\\[1mm]
 {\Large\bf boundary conditions in heterogeneous environments}\\[1mm]
Mingxin Wang\footnote{ {\sl E-mail}: mxwang@hpu.edu.cn. Mingxin Wang was supported by National Natural Science Foundation of China Grant 12171120.}\\
{\small School of Mathematics and Information Science, Henan Polytechnic University, Jiaozuo 454003, China}\\
Lei Zhang\footnote{The corresponding author, {\sl E-mail}: zhanglei890512@gmail.com. Lei Zhang was supported by the National Natural Science Foundation of China (12471168, 12171119) and the Fundamental Research Funds for the Central Universities (GK202304029, GK202306003, GK202402004). }\\
{\small School of Mathematics and Statistics, Shaanxi Normal University, Xi'an 710119, China}
\end{center}

\begin{quote}
\noindent{\bf Abstract.} To describe the propagation of West Nile virus and/or Zika virus, in this paper, we propose and study a time-periodic reaction-diffusion model with general boundary conditions in heterogeneous environments and with four unknowns: susceptible host, infectious host, susceptible vector and infectious vector. We can prove that such problem has a positive time periodic solution if and only if host and vector persist and the basic reproduction ratio is greater than one, and moreover the positive time periodic solution is unique and globally asymptotically stable when it exists.


\end{quote}

\pagestyle{myheadings}
\section{Introduction}{\setlength\arraycolsep{2pt}
\markboth{\rm$~$ \hfill A reaction-diffusion virus model\hfill $~$}{\rm$~$ \hfill M.X. Wang \& L. Zhang\hfill $~$}

Vector-host epidemic models are used to describe the dynamics of disease transmission between vectors (such as mosquitoes) and hosts (such as humans, birds). Such models usually take into account the transmission dynamics of the disease, including the infection, spread and recovery process of the disease. Vector-host epidemic models typically include reaction-diffusion equations that concern the spatial structure, i.e. the random movement of vector and host  populations over a geographic area, and can explain the geographic spread of disease, as well as the spatial spread of vector and host populations. In vector-host epidemic models, diseases usually spread between vectors and hosts in a cross-over fashion. That is, the infected vector transmits the disease to the susceptible host, and the susceptible vector becomes infected through interaction with the infected host.

Zika virus is one kind of host-vector disease, which is transmitted by mosquitoes, mainly through the bite of Aedes aegypti mosquitoes, and has become a significant public health threat due to its rapid spread and potential links to severe birth defects and neurological disorders. Fitzgibbon et al. \cite{FMW17} established a reaction-diffusion model to understand how spatial heterogeneity of the vector and host populations influences the dynamics of the outbreak, in both the geographical spread and the final size of the epidemic. Magal et al. \cite{MWW18} utilized the theory of asymptotically autonomous semiflows to conduct a detailed analysis of the autonomous reaction-diffusion model and proved that the basic reproduction ratio $R_0$ serves as a threshold value for the evolution dynamics of the model: if $R_0\le 1$ the disease free equilibrium is globally stable; if $R_0>1$ the model has a unique globally stable positive equilibrium. Considering the variation of seasonality, Li and Zhao \cite{LZ21} proposed a time periodic reaction-diffusion model for the Zika virus and studied its dynamics using chain transitive sets theory. Additionally, several reaction-diffusion models have been developed to describe the spatial spread of the Zika virus, including those in \cite{MLM19, Cao23}. Very recently, the first author of this paper and his collaborators \cite{Wang24, WZ24} employed the upper and lower solutions method to analyze these models. Compared to previous methods, the upper and lower solutions method can significantly shorten the length of the proof. Clearly, this approach is much more fundamental and reader-friendly.

West Nile (WN) virus is also one kind of host-vector disease. Mosquitoes transmit WNv to various hosts when they bite and feed on blood. Birds serve as the natural reservoirs for WNv, as they can be infected by the virus and can develop high viral titers in their blood, which can then infect mosquitoes that bite them (\cite{Cam02, Kom}). Wonham et al. \cite{WCL2004PRSLB} proposed an ODE model derived from the classical SIR model to investigate the dynamics of the West Nile virus. Lewis et al. \cite{LRD2006BMB} then studied the spread rates of the West Nile virus using a reaction-diffusion model.
The authors of \cite{Fan} proposed and studied a delayed differential equation model for the transmission of the West Nile virus between vectors (mosquitoes) and avian hosts (birds). Considering seasonal variations, Li et al. \cite{LLZ2020JNS} introduced a periodic time delays in their model and explored the dynamcs. There have been many investigations into the West Nile virus, including studies such as \cite{CCCetal2016DCDSB,LZ2017JMB}.

As far as we know, the previously focused host-vector models are either ordinary differential systems or reaction-diffusion systems in which the density of susceptible hosts do not change over time (constant, or a function of $x$). As pointed out by the authors in \cite{FMW17}, these models are applicable during the early phase of the epidemic, when the outbreak does not significantly alter the local geographic and demographic structure of the host population. To better understand the spread throughout the various stages of an epidemic, it is reasonable to consider that the density of susceptible hosts also changes over time, taking into account the birth and death of the hosts.

Based on the discussions of \cite{WCL2004PRSLB,LRD2006BMB,FMW17,LZ21}, we make the following  assumptions to understand how the spatiotemporal heterogeneity, birth, death, environmental crowding and boundary conditions of the vector and host populations influence the dynamics of the epidemic outbreak.\vspace{-1.5mm}
 \begin{enumerate}[$(1)$]
 \item The living environment of hosts and vectors is spatiotemporal heterogeneity and time periodic (seasonal reason).
 \item  Only one host species (humans, or birds) and one vector species (mosquitoes) are concerned.
   \item All uninfected hosts and vectors are susceptible.
 \item All infectious hosts and vectors give birth to uninfected offspring.
 \item Hosts and vectors carrying virus cannot lose its infection.
  \item The infected hosts and vector do not die of the
virus, and are not in any other way affected by the virus.
\item Pay attention to the diffusion, birth, death and environmental crowding (intraspecific competition) of hosts and vectors.
 \item Uninfected and infected species within the same species have the same diffusion patterns and boundary conditions.
 \vspace{-1.5mm}
\end{enumerate}

Let $H_u(x,t)$, $H_i(x,t)$, $V_u(x,t)$ and $V_i(x,t)$ be the densities of susceptible hosts, infected hosts, susceptible vectors, and infected vectors at location $x$ and time $t$, respectively. Then $H(x,t)=H_u(x,t)+H_i(x,t)$ and $V(x,t)=V_u(x,t)+V_i(x,t)$ are the total densities of hosts and vectors.
Denote by, respectively, $d_1(x,t)$ and $d_2(x,t)$ the host and vector diffusion rates at location $x$ and time $t$, $a_1(x,t)$ and $a_2(x,t)$ the birth rates of susceptible hosts and vectors, $b_1(x,t)$ and $b_2(x,t)$ the death rates of susceptible  hosts and vectors, $c_1(x,t)$ and $c_2(x,t)$ the loss rates of hosts and vectors population due to environmental crowding, $\ell_1(x,t)$ and $\ell_2(x,t)$ the transmission rates of uninfected host and uninfected vectors. The concerned model of this paper are the following reaction-diffusion system
\bes\left\{\begin{aligned}
	&\partial_t H_u=\nabla\cdot d_1(x,t)\nabla H_u+a_1(x,t)(H_u+H_i)-b_1(x,t)H_u\\
	&\hspace{17mm}-c_1(x,t)(H_u+H_i)H_u
	-\ell_1(x,t)H_uV_i,&&x\in\oo,\; t>0,\\
	&\partial_t H_i=\nabla\cdot d_1(x,t)\nabla H_i+\ell_1(x,t)H_uV_i-b_1(x,t)H_i-c_1(x,t)(H_u+H_i)H_i,\!\!&&x\in\oo,\; t>0,\\
	&\partial_t V_u=\nabla\cdot d_2(x,t)\nabla V_u+a_2(x,t)(V_u+V_i)-b_2(x,t)V_u\\
	&\hspace{17mm}
	-c_2(x,t)(V_u+V_i)V_u-\ell_2(x,t)H_iV_u,&&x\in\oo,\; t>0,\\
	&\partial_t V_i=\nabla\cdot d_2(x,t)\nabla V_i+\ell_2(x,t)H_iV_u-b_2(x,t)V_i-c_2(x,t)(V_u+V_i)V_i,&&x\in\oo,\; t>0,\\
	&\alpha_1\dd\frac{\partial H_u}{\partial\nu}+\beta_1(x,t)H_u=\alpha_1\dd\frac{\partial H_i}{\partial\nu}+\beta_1(x,t)H_i=0,&&x\in\partial\oo,\; t>0,\\[1mm]
	&\alpha_2\dd\frac{\partial V_u}{\partial\nu}+\beta_2(x,t)V_u=\alpha_2\dd\frac{\partial V_i}{\partial\nu}+\beta_2(x,t)V_i=0,&&x\in\partial\oo,\; t>0,\\
	&\big(H_u, H_i, V_u, V_i\big)=\big(H_{u0}(x), H_{i0}(x), V_{u0}(x), V_{i0}(x)),&&x\in\oo,\; t=0,\vspace{-2mm}
\end{aligned}\rr.\label{1.1}\ees
where $\nu$ is the outward normal vector of $\partial\Omega$, $\nabla\cdot d(x,t)\nabla U=\nabla\cdot[d(x,t)\nabla U]$, and the coefficient functions satisfy \vspace{-7pt}
\begin{enumerate}
	\item[{\bf(H)}]\, All coefficient functions are $T$-periodic in time $t$, H\"{o}lder continuous, nonnegative and nontrivial; $d_j(x,t), a_j(x,t), c_j(x,t)>0$ in $\overline Q_T$, and there exists $(x_0,t_0)\in\overline Q_T$ such that $\ell_j(x_0,t_0)>0$; and $\nabla_xd_j$ is H\"{o}lder continuous in $\overline Q_T$, $j=1,2$.
	\item[{\bf(B)}]\, Either
	
	(i) $\alpha_j\equiv 0$, $\beta_j(x,t)\equiv 1$ for $j=1,2$; or
	
	(ii) $\alpha_j\equiv 1$, $\beta_j(x,t)\geq 0$ with $\beta_j\in C^{1+\alpha,\, (1+\alpha)/2}(\overline S_T)$ for $j=1,2$.\vspace{-5pt}
\end{enumerate}
The initial functions $H_{u0}, H_{i0}, V_{u0}, V_{i0}\in C^1(\bar \Omega)$ be positive  and satisfy the compatibility conditions:
\bess
&\alpha_1\dd\frac{\partial H_{u0}}{\partial\nu}+\beta_1(x,0)H_{u0}=\alpha_1\frac{\partial H_{i0}}{\partial\nu}+\beta_1(x,0)H_{i0}=0,\;\;x\in\partial\oo,&\\[1.2mm]
&\alpha_2\dd\frac{\partial V_{u0}}{\partial\nu}+\beta_2(x,0)V_{u0}=\alpha_2\frac{\partial V_{i0}}{\partial\nu}+\beta_2(x,0)V_{i0}=0,\;\;x\in\partial\oo.&
\eess

In this paper, we study the model \eqref{1.1} and prove the existence, uniqueness and global stability of the positive $T$-periodic solution in time $t$, using the upper and lower solutions method rather than abstract chain transitive sets theory and asymptotically autonomous semi-flows employed in \cite{LZ21} and \cite{MWW18}.

In Section 2, we first determine the {\it Basic Reproduction Number} $\mathcal{R}_0$, and then prove that a positive time periodic solution exists if and only if $\mathcal{R}_0>1$, and that it is unique whenever it exists (Theorem \ref{th2.1}). In Section 3, we study the global asymptotic stabilities of nonnegative time periodic solutions. In Section 4, we present the dynamics of the autonomous case and provide explicit expressions for the constant case under Neumann boundary conditions. Additionally, we investigate the effect of parameters on the Basic Reproduction Number through numerical simulations in Section 5. The last section is a brief discussion. Some ideas of this paper come from \cite{Wang24} and \cite{WZ24}. For the sake of convenience, this paper always presents the results from the perspective of the West Nile virus model.

Throughout this paper, we denote
  \[Q_T=\oo\times(0,T],\;\;\;S_T=\partial\oo\times(0,T].\]

\section{Positive time periodic solutions}\setcounter{equation}{0} {\setlength\arraycolsep{2pt}

In this section, we will demonstrate the existence and uniqueness of positive solutions for the following time periodic problem associated to \qq{1.1}:
 \bes\left\{\begin{aligned}
&\partial_t{\mathsf H}_u=\nabla\cdot d_1(x,t)\nabla{\mathsf H}_u+
a_1(x,t)({\mathsf H}_u+{\mathsf H}_i)-b_1(x,t){\mathsf H}_u\\
&\hspace{17mm}-c_1(x,t)({\mathsf H}_u+{\mathsf H}_i){\mathsf H}_u
-\ell_1(x,t){\mathsf H}_u{\mathsf V}_i,&&(x,t)\in Q_T,\\
&\partial_t{\mathsf H}_i=\nabla\cdot d_1(x,t)\nabla{\mathsf H}_i+
\ell_1(x,t){\mathsf H}_u{\mathsf V}_i-b_1(x,t){\mathsf H}_i
-c_1(x,t)({\mathsf H}_u+{\mathsf H}_i){\mathsf H}_i,&&(x,t)\in Q_T,\\
&\partial_t{\mathsf V}_u=\nabla\cdot d_2(x,t)\nabla{\mathsf V}_u+
a_2(x,t)({\mathsf V}_u+{\mathsf V}_i)-b_2(x,t){\mathsf V}_u\\
&\hspace{17mm}-c_2(x,t)({\mathsf V}_u+{\mathsf V}_i){\mathsf V}_u
-\ell_2(x,t){\mathsf H}_i{\mathsf V}_u,&&(x,t)\in Q_T,\\
&\partial_t{\mathsf V}_i=\nabla\cdot d_2(x,t)\nabla{\mathsf V}_i+\ell_2(x,t){\mathsf H}_i{\mathsf V}_u-b_2(x,t) {\mathsf V}_i
-c_2(x,t)({\mathsf V}_u+{\mathsf V}_i){\mathsf V}_i,&&(x,t)\in Q_T,\\
&\alpha_1\dd\frac{\partial{\mathsf H}_u}{\partial\nu}+\beta_1(x,t){\mathsf H}_u
=\alpha_1\dd\frac{\partial{\mathsf H}_i}{\partial\nu}
+\beta_1(x,t){\mathsf H}_i=0,&&(x,t)\in S_T,\\[1mm]
&\alpha_2\dd\frac{\partial{\mathsf V}_u}{\partial\nu}+\beta_2(x,t){\mathsf V}_u
=\alpha_2\dd\frac{\partial{\mathsf V}_i}{\partial\nu}+\beta_2(x,t)
{\mathsf V}_i=0,&&(x,t)\in S_T,\\
&({\mathsf H}_u, {\mathsf H}_i, {\mathsf V}_u, {\mathsf V}_i)(x, 0)=({\mathsf H}_u, {\mathsf H}_i, {\mathsf V}_u, {\mathsf V}_i)(x, T),&&x\in\oo.
 \end{aligned}\rr.\label{2.1}\ees

We first consider the time periodic problem of logistic type equations, with $j=1,2$,
 \[\left\{\begin{array}{lll}
\partial_t{\mathsf U}-\nabla\cdot d_j(x,t)\nabla{\mathsf U}=\big(a_j(x,t)-b_j(x,t)\big){\mathsf U}
-c_j(x,t){\mathsf U}^2,\;&(x,t)\in Q_T,\\[2mm]
\alpha_j\dd\frac{\partial{\mathsf U}}{\partial\nu}+\beta_j(x,t){\mathsf U}=0,\;&(x,t)\in S_T,\\[2mm]
{\mathsf U}(x,0)={\mathsf U}(x,T),\;&x\in\oo.
 \end{array}\right.\eqno(2.2_j)\]
If $\big({\mathsf H}_u(x,t), {\mathsf H}_i(x,t), {\mathsf V}_u(x,t), {\mathsf V}_i(x,t)\big)$ is a nonnegative solution of \eqref{2.1}, then ${\mathsf H}(x,t)={\mathsf H}_u(x,t)+{\mathsf H}_i(x,t)$ and ${\mathsf V}(x,t)={\mathsf V}_u(x,t)+{\mathsf V}_i(x,t)$ satisfy $(2.2_1)$ and $(2.2_2)$, respectively.

Let $\zeta_j$ be the principal eigenvalue of
 \[\begin{cases}
\partial_t\varphi-\nabla\cdot d_j(x,t)\nabla\varphi+\big(b_j(x,t)-a_j(x,t)\big)\varphi=
 \zeta\varphi,\;&(x,t)\in Q_T,\\[1mm]
\alpha_j\dd\frac{\partial\varphi}{\partial\nu}+\beta_j(x,t)\varphi=0,\;&(x,t)\in S_T,\\[1mm]
\varphi(x,0)=\varphi(x,T),\;&x\in\oo
	\end{cases}\eqno(2.3_j)\]
with $j=1,2$, and define
 \bess
 Y_j=C(\overline\Omega,\mathbb{R})\;\;{\rm when}\;\;\alpha_j\equiv 1,\;\;\;
  Y_j=\{u\in C(\overline\Omega,\mathbb{R}):\,u|_{\partial \Omega}=0\}\;\;{\rm when}\;\;\alpha_j\equiv 0,\eess
and
  \bess
C_T(\mathbb{R},Y_j)=\{ \mathsf{U} \in  C(\mathbb{R},Y_j):\,\mathsf{U}  (t)=\mathsf{U} (t+T), ~ t\in \mathbb{R} \},\eess
which is equipped with the maximum norm.
Let $\{\Phi_j(t,s): t\ge s\}$ be the $T$-periodic evolution family on $Y_j$  determined by the following equation:
 \[\begin{cases}
\partial_t u-\nabla\cdot d_j(x,t)\nabla u +b_j(x,t)u=0,&x\in\Omega,\;t>0,\\[1mm]
\alpha_j\dd\frac{\partial u}{\partial\nu}+\beta_j(x,t)u=0,&x\in\partial\Omega,\;t>0,\\[1mm]
\end{cases}\]
that is, $\Phi_j(t,s)\phi=u_j(t,s,\cdot;\phi)$, where $u_j$ is the solution of the above equations with $u_j(s,s,x;\phi)=\phi(x)$.
Define
 \[F_j(t) \mathsf{U}(t) =a_j(\cdot,t) \mathsf{U}(t), \;\;~ \mathsf{U} \in C_T(\mathbb{R},Y_j).\]

In view of  $b_j(x,t)>0$, we have $\omega(\Phi_j)<0$, that is, the exponential growth bound of $\Phi_j$ are negative. We can define the following two operators:
 \bess
 L_j[\mathsf{U}](t)&=&\int_{-\infty}^{t} \Phi_j(t,s) F_j(s)\mathsf{U}(s) \mathrm{d}s,\\
 \hat L_j[\mathsf{U}](t)&=&F_j(t)\int_{-\infty}^{t} \Phi_j(t,s)\mathsf{U}(s) \mathrm{d}s\eess
with $\mathsf{U}\in C_T(\mathbb{R},Y_j)$, $j=1,2$,
and define the {\it Basic Reproduction Numbers} by
 \[\mathcal{R}_{0,j}=r(L_j)=r(\hat L_j).\]
Then $\mathcal{R}_{0,j}-1$ and $-\zeta_j$ share the same sign by \cite[Theorem 5.7]{thieme2009spectral} or  \cite[Theorem 3.7]{LZZ2019R0} with $\tau=0$.
According to \cite[Theorem 28.1]{Hess}, $(2.2_j)$ has no positive solution when $\mathcal{R}_{0,j} \le 1$. It concludes\vspace{-2mm} that
\begin{enumerate}
\item[{\rm(i)}]\, ${\mathsf H}(x,t)=0$, i.e., ${\mathsf H}_u(x,t)={\mathsf H}_i(x,t)=0$, and then ${\mathsf V}_i(x,t)=0$ by the equations of ${\mathsf V}_i$ when $\mathcal{R}_{0,1} \le 1$;
\item[{\rm(ii)}]\,${\mathsf V}(x,t)=0$, i.e., ${\mathsf V}_u(x,t)={\mathsf V}_i(x,t)=0$, and then ${\mathsf H}_i(x,t)=0$ by the equations of ${\mathsf H}_i$ when $\mathcal{R}_{0,2} \le 1 $.\vspace{-2mm}
\end{enumerate}

In the following we assume that $\mathcal{R}_{0,j}>1$, $j=1,2$. Then $(2.2_1)$ and $(2.2_2)$  have unique positive solutions ${\mathsf H}(x,t)$ and ${\mathsf V}(x,t)$, respectively, and ${\mathsf H}(x,t)$ and ${\mathsf V}(x,t)$ are globally asymptotically stable in the periodic sense.\setcounter{equation}{3}

Therefore, investigating the positive periodic solutions of \eqref{2.1} can be transformed into finding the positive periodic solutions $({\mathsf H}_i, {\mathsf V}_i)$ of
 \bes\left\{\begin{aligned}
&\partial_t{\mathsf H}_i=\nabla\cdot d_1(x,t)\nabla{\mathsf H}_i
+\ell_1(x,t)({\mathsf H}(x,t)-{\mathsf H}_i){\mathsf V}_i\\
&\hspace{17mm}
-\big[b_1(x,t)+c_1(x,t){\mathsf H}(x,t)\big]{\mathsf H}_i,&&(x,t)\in Q_T,\\
&\partial_t{\mathsf V}_i=\nabla\cdot d_2(x,t)\nabla{\mathsf V}_i+
\ell_2(x,t)({\mathsf V}(x,t)-{\mathsf V}_i){\mathsf H}_i\\
&\hspace{17mm}-\big[b_2(x,t)+c_2(x,t){\mathsf V}(x,t)\big]{\mathsf V}_i,&&(x,t)\in Q_T,\\
&\alpha_1\dd\frac{\partial{\mathsf H}_i}{\partial\nu}
+\beta_1(x,t){\mathsf H}_i=\alpha_2\dd\frac{\partial{\mathsf V}_i}
{\partial\nu}+\beta_2(x,t){\mathsf V}_i=0,&&(x,t)\in S_T,\\
&{\mathsf H}_i(x,0)={\mathsf H}_i(x,T),\;\; {\mathsf V}_i(x,0)={\mathsf V}_i(x,T),&&x\in\oo
 \end{aligned}\rr.\label{2.4}\ees
satisfying ${\mathsf H}_i<{\mathsf H}$, ${\mathsf V}_i<{\mathsf V}$ in $Q_T$.

We then consider the following eigenvalue problem:
 \bes\left\{\begin{aligned}
&\partial_t\phi_1-\nabla\cdot d_1(x,t)\nabla\phi_1-\ell_1(x,t){\mathsf H}(x,t)\phi_2\\
&\hspace{18mm}+\big[b_1(x,t)+c_1(x,t){\mathsf H}(x,t)\big]\phi_1
=\lm\phi_1,&&(x,t)\in Q_T,\\
&\partial_t\phi_2-\nabla\cdot d_2(x,t)\nabla\phi_2-\ell_2(x,t){\mathsf V}(x,t)\phi_1\\
&\hspace{18mm}
+\big[b_2(x,t)+c_2(x,t){\mathsf V}(x,t)\big]\phi_2=\lm\phi_2,&&(x,t)\in Q_T,\\
&\alpha_1\dd\frac{\partial\phi_1}{\partial\nu}+\beta_1(x,t)\phi_1=
\alpha_2\dd\frac{\partial\phi_2}{\partial\nu}+\beta_2(x,t)\phi_2=0,&&(x,t)\in S_T,\\
&\phi_1(x,0)=\phi_1(x,T), \ \ \phi_2(x,0)=\phi_2(x,T),&&x\in\oo,
 \end{aligned}\rr.\label{2.5}\ees
which is the eigenvalue problem associated with linearizing system for \qq{2.4} at $(0,0)$. By \cite[Theorem 2.1]{WZ24}, the problem \qq{2.5} has a unique principal eigenvalue $\lm({\mathsf H}, {\mathsf V})$ with positive eigenfunction $\phi=(\phi_1,\phi_2)^T$.
Define
 \[X=Y_1\times Y_2,\]
and
 \[C_T(\mathbb{R},X)=\{\mathsf{U}\in C(\mathbb{R},X):\mathsf{U}(t)=\mathsf{U}(t+T), ~ t\in \mathbb{R}\},\]
which is equipped with the maximum norm.

Let $\{\Phi(t,s): t\ge s\}$ be the $T$-periodic evolution family on $X$ determined by the following system:
 \bess\left\{\begin{aligned}
	&\partial_t u_1-\nabla\cdot d_1(x,t)\nabla u_1
	+\big[b_1(x,t)+c_1(x,t){\mathsf H}(x,t)\big] u_1
	=0,&&x\in\Omega,\;t>0,\\
	&\partial_t u_2-\nabla\cdot d_2(x,t)\nabla u_2
	+\big[b_2(x,t)+c_2(x,t){\mathsf V}(x,t)\big] u_2=0,&&x\in\Omega,\;t>0,\\
	&\alpha_1\dd\frac{\partial u_1}{\partial\nu}+\beta_1(x,t) u_1=
 \alpha_2\dd\frac{\partial u_2}{\partial\nu}
 +\beta_2(x,t)u_2=0,&&x\in\partial\Omega,\;t>0.
  \end{aligned}\rr.\eess
Define
 \bess
 F (t)\mathsf{U}(t)=\begin{pmatrix}
	l_1(\cdot,t)\mathsf{H}(\cdot,t)\mathsf{U}_2(t)\\
	l_2(\cdot,t)\mathsf{V}(\cdot,t)\mathsf{U}_1(t)
\end{pmatrix}, ~ \;\;\mathsf{U}=\begin{pmatrix}
 \mathsf{U}_1\\
 \mathsf{U}_2
  \end{pmatrix}\in C_T(\mathbb{R},X).\eess
It is not hard to verify that $\omega(\Phi)<0$. We can define the following two operators:
 \bess
\mathcal{L}[\mathsf{U}](t)&=&\int_{-\infty}^{t} \Phi(t,s) F(s) \mathsf{U}(s) \mathrm{d} s,\\
~ \hat{\mathcal{L}}[\mathsf{U}](t)&=&F(t)\int_{-\infty}^{t} \Phi(t,s) \mathsf{U}(s) \mathrm{d} s
  \eess
with $\mathsf{U} \in  C_T(\mathbb{R},X)$, and define the {\it Basic Reproduction Number} by
  \[\mathcal{R}_{0}=r(\mathcal{L})=r(\hat{\mathcal{L}}).\]
Then $\mathcal{R}_{0}-1$ and $-\lm({\mathsf H}, {\mathsf V})$ have the same sign by \cite[Theorem 5.7]{thieme2009spectral} or  \cite[Theorem 3.7]{LZZ2019R0} with $\tau=0$. Now we can present the main result of this section.

\begin{theo}\lbl{th2.1} Assume that $\mathcal{R}_{0,j}>1$, $j=1,2$.  Then problem \eqref{2.4} has a positive solution $({\mathsf H}_i, {\mathsf V}_i)$ if and only if $\mathcal{R}_0>1$. Moreover, the positive solution $({\mathsf H}_i, {\mathsf V}_i)$ of \eqref{2.4} is unique and satisfies
 \[{\mathsf H}_i<{\mathsf H},\;\;\;{\mathsf V}_i<{\mathsf V}
 \;\;\;\mbox{in}\;\; Q_T\]
when it exists. Therefore, if $\mathcal{R}_0>1$, then \eqref{2.1} has a unique positive solution
 \[({\mathsf H}_u, {\mathsf H}_i, {\mathsf V}_u, {\mathsf V}_i)=({\mathsf H}-{\mathsf H}_i, {\mathsf H}_i, {\mathsf V}-{\mathsf V}_i, {\mathsf V}_i).\]
 \end{theo}

\begin{proof}\;  {\bf The necessity}. We rewrite \eqref{2.4} as follows:
	 \bess\left\{\begin{aligned}
	&\partial_t{\mathsf H}_i-\nabla\cdot d_1(x,t)\nabla{\mathsf H}_i -p_{11}(x,t){\mathsf H}_i-p_{12}(x,t){ \mathsf V}_i - \lm({\mathsf H}, {\mathsf V}){\mathsf H}_i\\
	&\hspace{37mm}=-\ell_1(x,t){\mathsf H}{\mathsf V}_i -\lm({\mathsf H}, {\mathsf V}){\mathsf H}_i,
	&&(x,t)\in Q_T,\\
	&\partial_t{\mathsf V}_i-\nabla\cdot d_2(x,t)\nabla{\mathsf V}_i -p_{21}(x,t) {\mathsf H}_i -p_{22}(x,t){\mathsf V}_i- \lm({\mathsf H}, {\mathsf V}){\mathsf V}_i\\
	&\hspace{37mm}=-\ell_2(x,t){\mathsf V}{\mathsf H}_i -\lm({\mathsf H}, {\mathsf V}){\mathsf V}_i,&&(x,t)\in Q_T,\\
	&\alpha_1\dd\frac{\partial{\mathsf H}_i}{\partial\nu}
+\beta_1(x,t){\mathsf H}_i=\alpha_2\dd\frac{\partial{\mathsf V}_i}
{\partial\nu}+\beta_2(x,t){\mathsf V}_i=0,&&(x,t)\in S_T,\\
&{\mathsf H}_i(x,0)={\mathsf H}_i(x,T),\;\; {\mathsf V}_i(x,0)={\mathsf V}_i(x,T),&&x\in\oo
	\end{aligned}\rr.\eess
where $p_{11}(x, t)=b_1(x,t)+c_1(x,t){\mathsf H}(x,t)$, $p_{12}(x,t)=\ell_1(x, t){\mathsf H}(x,t)$, $p_{21}(x, t)=\ell_2(x,t){\mathsf V} (x,t)$, $p_{22}(x, t)=b_2(x,t)+c_2(x,t){\mathsf V} (x,t)$.
		
We first consider the case where $\alpha_1=\alpha_2=1$. Let $\{Q(t,s):t \ge s\}$ be the $T$-periodic evolution family on $C(\overline{\Omega},\mathbb{R}^2)$ determined by the following system
 	\bes\left\{\begin{aligned}
	&\partial_t{\mathsf H}_i-\nabla\cdot d_1(x,t)\nabla{\mathsf H}_i -p_{11}(x,t){\mathsf H}_i-p_{12}(x,t){\mathsf V}_i - \lm({\mathsf H}, {\mathsf V}){\mathsf H}_i=0,
&&x\in\Omega,\;t>0,\\
&\partial_t{\mathsf V}_i-\nabla\cdot d_2(x,t)\nabla{\mathsf V}_i -p_{21}(x,t) {\mathsf H}_i -p_{22}(x,t){\mathsf V}_i- \lm({\mathsf H}, {\mathsf V}){\mathsf V}_i=0,&&x\in\Omega,\;t>0,\\
&\alpha_1\dd\frac{\partial{\mathsf H}_i}{\partial\nu}
+\beta_1(x,t){\mathsf H}_i=\alpha_2\dd\frac{\partial{\mathsf V}_i}
{\partial\nu}+\beta_2(x,t){\mathsf V}_i=0,&&x\in\partial\Omega,\;t>0.
 \end{aligned}\rr.
  \label{2.6}\ees
Then the Poincar\'{e} map $Q(T,0)$ is strongly positive and compact on $C(\overline{\Omega},\mathbb{R}^2)$. Since $\phi=(\phi_1,\phi_2)>0$ satisfies \qq{2.5}, we have $\phi(\cdot,0)=\phi(\cdot,T)=Q(T,0)\phi(\cdot,0)$. It then follows that $r\big(Q(T,0)\big)=1$. Denote $\Psi(\cdot,t)=\big({\mathsf H}_i(\cdot,t), {\mathsf V}_i(\cdot,t)\big)$. We then have
  \bess
r(Q(T,0)\big)\Psi(\cdot,0)=\Psi(\cdot,0)=\Psi(\cdot,T)=Q(T,0)\Psi(\cdot,0)-\int_0^T Q(T,s)G(s;\Psi(\cdot,s)\big){\rm d}s.
 \eess
where
  \[G(s;\Psi(\cdot,s))=\left(\begin{array}{cc}\ell_1(\cdot,s){\mathsf H}_i(\cdot,s){\mathsf V}_i(\cdot,s)+\lm({\mathsf H}, {\mathsf V}){\mathsf H}_i(\cdot,s)\\[1mm] \ell_2(\cdot,s){\mathsf V}_i(\cdot,s){\mathsf H}_i(\cdot,s)+\lm({\mathsf H}, {\mathsf V}){\mathsf V}_i(\cdot,s) \end{array}\right). \]
We assume on the contradiction that $\mathcal{R}_0\le1$. Then $\lm({\mathsf H}, {\mathsf V})\geq 0$, and thus
  \[G(s;\Psi(\cdot,s))=\left(\begin{array}{cc} \ell_1(\cdot,s){\mathsf H}_i(\cdot,s){\mathsf V}_i(\cdot,s)+\lm({\mathsf H}, {\mathsf V}){\mathsf H}_i(\cdot,s)\\[1mm] \ell_2(\cdot,s){\mathsf V}_i(\cdot,s){\mathsf H}_i(\cdot,s)+\lm({\mathsf H}, {\mathsf V}){\mathsf V}_i(\cdot,s) \end{array}\right)\ge
\left(\begin{array}{cc} \ell_1(\cdot,s){\mathsf H}_i(\cdot,s){\mathsf V}_i(\cdot,s)\\[1mm] \ell_2(\cdot,s){\mathsf V}_i(\cdot,s){\mathsf H}_i(\cdot,s) \end{array}\right). \]
Noticing that $\ell_j(\cdot,s){\mathsf V}_i(\cdot,s){\mathsf H}_i(\cdot,s)\ge 0,\,\not\equiv 0$, $j=1,2$. This is impossible by the conclusion \cite[Theorem 3.2 (iv)]{Am76}.

For the case $\alpha_1=\alpha_2=0$, $\{Q(t,s):t \ge s\}$ is the $T$-periodic evolution family on $C_0^1(\overline{\Omega},\mathbb{R}^2)$ determined by \eqref{2.6}. The Poincar\'{e} map $Q(T,0)$ is strongly positive and compact on $C_0^1(\overline{\Omega},\mathbb{R}^2)$. The remaining parts of the proof are similar.

{\bf The sufficiency}. Assume $\mathcal{R}_0>1$, which is equivalent to $\lm({\mathsf H}, {\mathsf V})<0$. We shall show that \qq{2.4} has a unique positive solution $({\mathsf H}_i, {\mathsf V}_i)$, and ${\mathsf H}_i<{\mathsf H}$, ${\mathsf V}_i<{\mathsf V}$. For this purpose and for subsequent applications, we first prove the following general \vspace{-4mm} conclusion.
  \begin{enumerate}[$(1)$]
\item[{\bf\quad Conclusion-$\ep$}:]\; Let $\varphi_j(x,t)$ be the positive eigenfunction corresponding to $\zeta_j$ of $(2.3_j)$. Normalize $\varphi_j$ by $\|\varphi_j\|_{L^{\infty}(Q_T)}=1$. Then there is $0<\ep_0\ll 1$ such that, when $|\ep|\le\ep_0$, problems
 \bes\left\{\begin{aligned}
&\partial_t{\mathsf H}_i=\nabla\cdot d_1(x,t)\nabla{\mathsf H}_i+\ell_1(x,t)\big({\mathsf H}(x,t)
+\ep\varphi_1(x,t)-{\mathsf H}_i\big)^+ {\mathsf V}_i\\
 &\hspace{17mm}-\big[b_1(x,t)+c_1(x,t)\big({\mathsf H}(x,t)-\ep\varphi_1(x,t)\big)\big]
 {\mathsf H}_i,\!\!&&(x,t)\in Q_T,\\
&\partial_t{\mathsf V}_i=\nabla\cdot d_2(x,t)\nabla{\mathsf V}_i+\ell_2(x,t)\big({\mathsf V}(x,t)
+\ep\varphi_2(x,t)-{\mathsf V}_i\big)^+ {\mathsf H}_i\\
&\hspace{17mm}-\big[b_2(x,t)+c_2(x,t)\big({\mathsf V}(x,t)-\ep\varphi_2(x,t)\big)\big]
 {\mathsf V}_i,&&(x,t)\in Q_T,\\
 &\alpha_1\dd\frac{\partial{\mathsf H}_i}{\partial\nu}+\beta_1(x,t){\mathsf H}_i=
\alpha_2\dd\frac{\partial{\mathsf V}_i}{\partial\nu}+\beta_2(x,t){\mathsf V}_i=0,
&&(x,t)\in S_T,\\[.1mm]
 &{\mathsf H}_i(x, 0)={\mathsf H}_i(x, T),\;\;\; {\mathsf V}_i(x, 0)={\mathsf V}_i(x, T),&&x\in\Omega
 \end{aligned}\rr.\label{2.7}\ees
and
  \bes\left\{\begin{aligned}
&\partial_t{\mathsf H}_i=\nabla\cdot d_1(x,t)\nabla{\mathsf H}_i+\ell_1(x,t)\big({\mathsf H}(x,t)
+\ep\varphi_1(x,t)-{\mathsf H}_i) {\mathsf V}_i\\
&\hspace{17mm}-\big[b_1(x,t)+c_1(x,t)\big({\mathsf H}(x,t)-\ep\varphi_1(x,t)\big)\big]{\mathsf H}_i ,&&(x,t)\in Q_T,\\
&\partial_t{\mathsf V}_i=\nabla\cdot d_2(x,t)\nabla{\mathsf V}_i+\ell_2(x,t)\big({\mathsf V}(x,t)
+\ep\varphi_2(x,t)-{\mathsf V}_i){\mathsf H}_i\\
 &\hspace{17mm}-\big[b_2(x,t)+c_2(x,t)\big({\mathsf V}(x,t)-\ep\varphi_2(x,t)\big)\big]
 {\mathsf V}_i,&&(x,t)\in Q_T,\\
 &\alpha_1\dd\frac{\partial{\mathsf H}_i}{\partial\nu}+\beta_1(x,t){\mathsf H}_i=\alpha_2\dd\frac{\partial{\mathsf V}_i}{\partial\nu}+\beta_2(x,t){\mathsf V}_i=0,
&&(x,t)\in S_T,\\[.1mm]
 &{\mathsf H}_i(x, 0)={\mathsf H}_i(x, T),\;\;\; {\mathsf V}_i(x, 0)={\mathsf V}_i(x, T),&&x\in\Omega
 \vspace{-2mm}\end{aligned}\rr.\label{2.8}\ees
 have, respectively, unique positive solutions $(\widehat{{\mathsf H}}_i^\ep, \widehat{{\mathsf V}}_i^\ep)$ and $({\mathsf H}_i^\ep, {\mathsf V}_i^\ep)$. Moreover,
 \bess
 \widehat{{\mathsf H}}_i^\ep,\,{\mathsf H}_i^\ep<{\mathsf H}+\ep\varphi_1, \;\;\;\widehat{{\mathsf V}}_i^\ep,\,{\mathsf V}_i^\ep<{\mathsf V}+\ep\varphi_2\;\;\;\mbox{in}\;\; Q_T.
 \eess
Therefore, $(\widehat{{\mathsf H}}_i^\ep, \widehat{{\mathsf V}}_i^\ep)=({\mathsf H}_i^\ep, {\mathsf V}_i^\ep)$ in $Q_T$,
and \qq{2.7} and \qq{2.8}\vspace{-1mm} are equivalent.
\end{enumerate}

Here, we only discuss the problem \eqref{2.7} as \qq{2.8} can be addressed in a similar manner.

{\it Existence of positive solutions of \qq{2.7}}.
Let $\lm({\mathsf H}, {\mathsf V}; \ep)$ be the principal eigenvalue of
\bes\left\{\begin{aligned}
&\partial_t\phi_1-\nabla\cdot d_1(x,t)\nabla\phi_1- \ell_1(x,t)\big[{\mathsf H}(x,t)+\ep\varphi_1(x,t)\big]\phi_2\nm\\
&\hspace{12mm}+\big[b_1(x,t)+c_1(x,t)\big({\mathsf H}(x,t)-\ep\varphi_1(x,t)\big)\big]\phi_1
 =\lm\phi_1,&&(x,t)\in Q_T,\\
&\partial_t\phi_2-\nabla\cdot d_2(x,t)\nabla\phi_2-\ell_2(x,t)
\big[{\mathsf V}(x,t)+\ep\varphi_2(x,t)\big]\phi_1\nm\\ &\hspace{12mm}+\big[b_2(x,t)+c_2(x,t)\big({\mathsf V}(x,t)-\ep\varphi_2(x,t)\big)\big]\phi_2
=\lm\phi_2,&&(x,t)\in Q_T,\\[.1mm]
&\alpha_1\dd\frac{\partial\phi_1}{\partial\nu}+\beta_1(x,t)\phi_1=
\alpha_2\dd\frac{\partial\phi_2}{\partial\nu}+\beta_2(x,t)\phi_2=0,&&(x,t)\in S_T,\\
 &\phi_1(x,0)=\phi_1(x,T), \;\; \phi_2(x,0)=\phi_2(x,T),&&x\in\Omega.
\vspace{-2mm}\end{aligned}\rr.\ees
In view of $\lm({\mathsf H}, {\mathsf V})<0$, there exists a $0<\ep_0\ll 1$ such that $\lm({\mathsf H}, {\mathsf V}; \ep)<0$ for all $|\ep|\le\ep_0$ by the continuity of $\lm({\mathsf H}, {\mathsf V}; \ep)$ in $\ep$, and
 \bes\left\{\begin{aligned}
&{\mathsf H}(x,t)\pm\ep\varphi_1(x,t)>0, \;\; {\mathsf V}(x,t)\pm\ep\varphi_2(x,t)>0,&&(x,t)\in Q_T, \\[0.5mm]
 &a_1(x,t)[{\mathsf H}(x,t)+\ep\varphi_1(x,t)]>\ep\varphi_1(x,t)(\ep c_1(x,t)\varphi_1(x,t)-\zeta_1),&&(x,t)\in Q_T,\\[0.5mm]
&a_2(x,t)[{\mathsf V}(x,t)+\ep\varphi_2(x,t)]>\ep\varphi_2(x,t)(\ep c_2(x,t)\varphi_2(x,t)-\zeta_2),&&(x,t)\in Q_T.
 \end{aligned}\rr.\label{2.9}\ees
In the case where $\alpha_1 =1$ and $\alpha_2 =1$, \qq{2.9} can be derived directly. For the case where $\alpha_1=0$ or $\alpha_2=0$, \qq{2.9} can be obtained using the method employed to prove the following \qq{equ:main:step1}. We omit the details here, as the proof of \qq{2.9} is simpler than that of \qq{equ:main:step1}.

Making use of \qq{2.9} we can show that $({\mathsf H}+\ep\varphi_1, {\mathsf V}+\ep\varphi_2)$ is a strict upper solution of \qq{2.7}.
Let $(\phi^\ep_1,\phi^\ep_2)$ be the positive eigenfunction corresponding to $\lm({\mathsf H}, {\mathsf V}; \ep)$. It is easy to verify that $\delta(\phi^\ep_1, \phi^\ep_2)$ is a lower solution of \qq{2.7} and 
 \bess
 \delta\big(\phi^\ep_1(x,t), \phi^\ep_2(x,t)\big)\le\big({\mathsf H}(x,t)+\ep\varphi_1(x,t), {\mathsf V}(x,t)+\ep\varphi_2(x,t)\big),\;\;\forall\;(x,t)\in Q_T
 \eess
provided that $\delta>0$ is suitably small. By the upper and lower solutions method (\cite[Theorem 7.15]{Wpara}), \qq{2.7} has at least one positive solution $(\widehat{{\mathsf H}}_i^\ep, \widehat{{\mathsf V}}_i^\ep)$, and
 \[\widehat{{\mathsf H}}_i^\ep(x,t)<{\mathsf H}(x,t)+\ep\varphi_1(x,t),\;\;\;
 \widehat{{\mathsf V}}_i^\ep(x,t)<{\mathsf V}(x,t)+\ep\varphi_2(x,t),\;\;\forall\;(x,t)\in Q_T.\]

{\it Uniqueness of positive solutions of \qq{2.7}}. Let $(\widetilde{{\mathsf H}}_i^\ep, \widetilde{{\mathsf V}}_i^\ep)$ be another positive solution of \qq{2.7}. We can find a constant $0<s<1$ such that $s(\widetilde{{\mathsf H}}_i^\ep, \widetilde{{\mathsf V}}_i^\ep)\le(\widehat{{\mathsf H}}_i^\ep, \widehat{{\mathsf V}}_i^\ep)$ in $Q_T$. Set
  \[\bar s=\sup\kk\{0<s\le 1: s(\widetilde{{\mathsf H}}_i^\ep, \widetilde{{\mathsf V}}_i^\ep)\le(\widehat{{\mathsf H}}_i^\ep, \widehat{{\mathsf V}}_i^\ep) \;{\rm ~ in  ~ } Q_T\rr\}.\]
Then $\bar s$ is well defined and $0<\bar s\le1$ and $\bar s(\widetilde{{\mathsf H}}_i^\ep, \widetilde{{\mathsf V}}_i^\ep)\le(\widehat{{\mathsf H}}_i^\ep, \widehat{{\mathsf V}}_i^\ep)$ in $Q_T$. We shall prove $\bar s=1$. We assume on the contrary that $\bar s<1$. Then ${\mathsf U}:=\widehat{{\mathsf H}}_i^\ep-\bar s\widetilde{{\mathsf H}}_i^\ep\ge 0$ and ${\mathsf Z}:=\widehat{{\mathsf V}}_i^\ep-\bar s\widetilde{{\mathsf V}}_i^\ep\ge 0$, as well as  $\bar s\widetilde{{\mathsf H}}_i^\ep\le\widehat{{\mathsf H}}_i^\ep<{\mathsf H}+\ep\varphi_1$ and $\bar s\widetilde{{\mathsf V}}_i^\ep\le\widehat{{\mathsf V}}_i^\ep<{\mathsf V}+\ep\varphi_2$ in $Q_T$. It follows that, in $Q_T$, 
 \bess
 &({\mathsf H}+\ep\varphi_1-\widehat{{\mathsf H}}_i^\ep)^+
 ={\mathsf H}+\ep\varphi_1-\widehat{{\mathsf H}}_i^\ep,&\\[0.5mm]
 &({\mathsf H}+\ep\varphi_1-\widetilde{{\mathsf H}}_i^\ep)^+
 <{\mathsf H}+\ep\varphi_1-\bar s\widetilde{{\mathsf H}}_i^\ep,&\\[0.5mm]
& ({\mathsf H}+\ep\varphi_1-\widehat{{\mathsf H}}_i^\ep)^+
-({\mathsf H}+\ep\varphi_1-\widetilde{{\mathsf H}}_i^\ep)^+
 >\bar s\widetilde{{\mathsf H}}_i^\ep-\widehat{{\mathsf H}}_i^\ep,&\\[0.5mm]
&({\mathsf V}+\ep\varphi_2-\widehat{{\mathsf V}}_i^\ep)^+
={\mathsf V}+\ep\varphi_2-\widehat{{\mathsf V}}_i^\ep,&\\[0.5mm] &({\mathsf V}+\ep\varphi_2-\widetilde{{\mathsf V}}_i^\ep)^+
 <{\mathsf V}+\ep\varphi_2-\bar s\widetilde{{\mathsf V}}_i^\ep,&\\[0.5mm]
&({\mathsf V}+\ep\varphi_2-\widehat{{\mathsf V}}_i^\ep)^+
-({\mathsf V}+\ep\varphi_2-\widetilde{{\mathsf V}}_i^\ep)^+
>\bar s\widetilde{{\mathsf V}}_i^\ep-\widehat{{\mathsf V}}_i^\ep.&
 \eess
After careful calculation, we can show that $({\mathsf U}, {\mathsf Z})$ satisfies
 \bes\left\{\begin{aligned}
&\partial_t{\mathsf U}>\nabla\cdot d_1(x,t)\nabla{\mathsf U}+\ell_1(x,t)\big[{\mathsf H}(x,t)
+\ep\varphi_1(x,t)-\widehat{{\mathsf H}}_i^\ep(x,t)\big]{\mathsf Z}\\[0.5mm]
 &\hspace{13mm}-\Big(b_1(x,t)+c_1(x,t)\big[{\mathsf H}(x,t)-\ep\varphi_1(x,t)\big]
+ \ell_1(x,t)\widetilde{{\mathsf V}}_i^\ep(x,t)\Big){\mathsf U},\;\;\;&&(x,t)\in Q_T,\\[0.5mm]
&\partial_t{\mathsf Z}>\nabla\cdot d_2(x,t)\nabla{\mathsf Z}+\ell_2(x,t)\big[{\mathsf V}(x,t)
+\ep\varphi_2(x,t)-\widehat{{\mathsf V}}_i^\ep(x,t)\big]{\mathsf U}\\[0.5mm]
 &\hspace{13mm}-\Big(b_2(x,t)+c_2(x,t)\big[{\mathsf V}(x,t)-\ep\varphi_2(x,t)\big]
+ \ell_2(x,t)\widetilde{{\mathsf H}}_i^\ep(x,t)\Big){\mathsf Z},&&(x,t)\in Q_T,\\[0.5mm]
&\alpha_1\dd\frac{\partial{\mathsf U}}{\partial\nu}+\beta_1(x,t){\mathsf U}=\alpha_2\dd\frac{\partial{\mathsf Z}}{\partial\nu}+\beta_2(x,t){\mathsf Z}=0,&&(x,t)\in S_T,\\[0.5mm]
&{\mathsf U}(x,0)={\mathsf U}(x,T),\ {\mathsf Z}(x,0)={\mathsf Z}(x,T),&&x\in\Omega.
 \end{aligned}\right.\qquad\lbl{2.10}\ees
In view of ${\mathsf H}+\ep\varphi_1-\widehat{{\mathsf H}}_i^\ep>0$, ${\mathsf V}+\ep\varphi_2-\widehat{{\mathsf V}}_i^\ep>0$ and ${\mathsf U}, {\mathsf Z}\ge 0$ in $Q_T$. It follows that ${\mathsf U},{\mathsf Z}>0$ in $Q_T$ by the maximum principle. Then there exists $0<\tau<1-\bar s$ such that $({\mathsf U}, {\mathsf Z})\ge\tau(\widetilde{{\mathsf H}}_i^\ep, \widetilde{{\mathsf V}}_i^\ep)$, i.e., $(\bar s+\tau)(\widetilde{{\mathsf H}}_i^\ep, \widetilde{{\mathsf V}}_i^\ep)\le(\widehat{{\mathsf H}}_i^\ep, \widehat{{\mathsf V}}_i^\ep)$ in $\overline Q_T$. This contradicts the definition of $\bar s$. Hence $\bar s=1$, i.e.,
 \[(\widetilde{{\mathsf H}}_i^\ep,\, \widetilde{{\mathsf V}}_i^\ep)\le(\widehat{{\mathsf H}}_i^\ep,\, \widehat{{\mathsf V}}_i^\ep)\;\;\;\mbox{in}\;\; Q_T.\]
Certainly, $\widetilde{{\mathsf H}}_i^\ep<{\mathsf H}+\ep\varphi_1$, $\widetilde{{\mathsf V}}_i^\ep<{\mathsf V}+\ep\varphi$, and then
 \bess
 ({\mathsf H}+\ep\varphi_1-\widetilde{{\mathsf H}}_i^\ep)^+
 ={\mathsf H}+\ep\varphi_1-\widetilde{{\mathsf H}}_i^\ep,\;\;\;
 ({\mathsf V}+\ep\varphi_2-\widetilde{{\mathsf V}}_i^\ep)^+
 ={\mathsf V}+\ep\varphi_2-\widetilde{{\mathsf V}}_i^\ep\;\;\;\mbox{in}\;\;Q_T.\eess

On the other hand, we can find $k>1$ such that
 \[k\big(\widetilde{{\mathsf H}}_i^\ep(x,t), \widetilde{{\mathsf V}}_i^\ep(x,t)\big)\ge\big(\widehat{{\mathsf H}}_i^\ep(x,t), \widehat{{\mathsf V}}_i^\ep(x,t)\big),\;\;\forall\;(x,t)\in Q_T.\]
Set
  \[\ud k=\inf\kk\{k\ge1: k(\widetilde{{\mathsf H}}_i^\ep, \widetilde{{\mathsf V}}_i^\ep)\ge(\widehat{{\mathsf H}}_i^\ep, \widehat{{\mathsf V}}_i^\ep)\;{\rm ~ in  ~ }\, Q_T\rr\}.\]
Then $\ud k$ is well defined, $\ud k\ge1$ and $\ud k(\widetilde{{\mathsf H}}_i^\ep, \widetilde{{\mathsf V}}_i^\ep)\ge(\widehat{{\mathsf H}}_i^\ep, \widehat{{\mathsf V}}_i^\ep)$ in $\overline Q_T$. If $\ud k>1$, then ${\mathsf P}:=\ud k\widetilde{{\mathsf H}}_i^\ep-\widehat{{\mathsf H}}_i^\ep\ge 0$ and ${\mathsf Q}:=\ud k\widetilde{{\mathsf V}}_i^\ep-\widehat{{\mathsf V}}_i^\ep\ge 0$. Similarly to the above, we can verify that $({\mathsf P}, {\mathsf Q})$ satisfies a system of differential inequalities similar to \qq{2.10} and derive ${\mathsf P}, {\mathsf Q}>0$ in $\oo\times[0,T]$ by the maximum principle, and there exists $0<r<\ud k-1$ such that $({\mathsf P}, {\mathsf Q})\ge r(\widetilde{{\mathsf H}}_i^\ep, \widetilde{{\mathsf V}}_i^\ep)$, i.e., $(\ud k-r)(\widetilde{{\mathsf H}}_i^\ep, \widetilde{{\mathsf V}}_i^\ep)\ge(\widehat{{\mathsf H}}_i^\ep, \widehat{{\mathsf V}}_i^\ep)$ in $Q_T$. This contradicts the definition of $\ud k$. Hence $\ud k=1$ and
 \[\big(\widetilde{{\mathsf H}}_i^\ep(x,t), \widetilde{{\mathsf V}}_i^\ep(x,t)\big)\ge\big(\widehat{{\mathsf H}}_i^\ep(x,t), \widehat{{\mathsf V}}_i^\ep(x,t)\big),\;\;\forall\;(x,t)\in\overline Q_T.\]
The uniqueness is obtained.

Taking $\ep=0$ in {\bf Conclusion-$\ep$}, the sufficiency and uniqueness are proved.
\end{proof}

\section{Dynamical properties of \qq{1.1}}\lbl{s3}\setcounter{equation}{0}

In this section we study the stabilities of nonnegative time periodic solutions. We first state the existence and uniqueness of global solutions.

\begin{theo}\lbl{th3.a} The problem \qq{1.1} has a unique global solution $(H_u, H_i, V_u, V_i)$, which is positive and bounded.
\end{theo}

\begin{proof} The proof is standard. For reader's convenience, we provide an outline of the proof. Firstly, the local existence can be proved by the upper and lower solutions method. The positivity and uniqueness  are a direct application of the maximum principle for parabolic equations. Let $T_{\rm max}$ be the maximum existence time of $(H_u, H_i, V_u, V_i)$ and set \[H(x,t)=H_u(x,t)+H_i(x,t),\;\;\;V(x,t)=V_u(x,t)+V_i(x,t).\]
Then $H$ and $V$ satisfy
 \bes\left\{\begin{aligned}
 &\partial_t H=\nabla\cdot d_1(x,t)\nabla H+\big(a_1(x,t)
 -b_1(x,t)\big)H-c_1(x,t)H^2,&&x\in\oo,\; 0<t<T_{\rm max},\\
&\alpha_1\dd\frac{\partial H}{\partial\nu}+\beta_1(x,t)H=0,&&x\in\partial\oo,\; 0<t<T_{\rm max},\\
&H(x,0)=H_u(x,0)+H_i(x,0)>0,&&x\in\oo
	\end{aligned}\rr.\lbl{3.1}\ees
and
 \bes\left\{\begin{aligned}
&\partial_t V=\nabla\cdot d_2(x,t)\nabla V+\big(a_2(x,t)
 -b_2(x,t)\big)V-c_2(x,t)V^2,&&x\in\oo,\;  0<t<T_{\rm max},\\
&\alpha_2\dd\frac{\partial V}{\partial\nu}+\beta_2(x,t)V=0,&&x\in\partial\oo,\; 0<t<T_{\rm max},\\
&V(x,0)=V_u(x,0)+V_i(x,0)>0,&&x\in\oo,
	\end{aligned}\rr.\lbl{3.2}\ees
respectively. By the maximum principle, $H$ and $V$ exist globally and are bounded, i.e., $T_{\rm max}=+\yy$. Consequently, $(H_u, H_i, V_u, V_i)$ exists globally and is positive and bounded.
\end{proof}

\begin{theo}\lbl{th3.1} Let $(H_u, H_i, V_u, V_i)$ be the unique positive solution of \qq{1.1}. \vspace{-2mm}
\begin{enumerate}[$(1)$]
\item\; In the case of $\mathcal{R}_{0,1}, \mathcal{R}_{0,2}>1$ and $\mathcal{R}_0>1$, we have that, in $[C^{2,1}(\overline Q_T)]^4$,
 \bes
&&\lim_{n\to+\yy}\big(H_u(x,t+nT),\, H_i(x,t+nT),\, V_u(x,t+nT),\, V_i(x,t+nT)\big)\nonumber\\[1mm]
&=&\big({\mathsf H}(x,t)-{\mathsf H}_i(x,t),\,{\mathsf H}_i(x,t),\, {\mathsf V}(x,t)-{\mathsf V}_i(x,t),\,{\mathsf V}_i(x,t)\big),
 \lbl{3.3}\ees
 where ${\mathsf H}(x,t)$, ${\mathsf V}(x,t)$ and $\big({\mathsf H}_i(x,t), {\mathsf V}_i(x,t)\big)$ are the unique positive solutions of $(2.2_1)$, $(2.2_2)$ and \qq{2.4}, respectively, and ${\mathsf H}_i(x,t)<{\mathsf H}(x,t), {\mathsf V}_i(x,t)<{\mathsf V}(x,t)$.

\item\; For the case of $\mathcal{R}_{0,1}, \mathcal{R}_{0,2}>1$ and $\mathcal{R}_0 \le 0$, we have that, in $[C^{2,1}(\overline Q_T)]^4$,
 \bes
\lim_{n\to+\yy}\big(H_u(x,t+nT),\,H_i(x,t+nT),\, V_u(x,t+nT),\, V_i(x,t+nT)\big)=\big({\mathsf H}(x,t), 0, {\mathsf V}(x,t), 0\big).
 \qquad\lbl{3.4}\ees

\item\; For the case of $\mathcal{R}_{0,1} \le 1$ and $\mathcal{R}_{0,2}>1$, we have  that, in $[C^{2,1}(\overline Q_T)]^4$,
 \bess
\lim_{n\to+\yy}\big(H_u(x,t+nT),\,H_i(x,t+nT),\, V_u(x,t+nT),\, V_i(x,t+nT)\big)=\big(0, 0, {\mathsf V}(x,t), 0\big).\eess

For the case of $\mathcal{R}_{0,2} \le 1 $ and $\mathcal{R}_{0,1}>1$, we have that, in $[C^{2,1}(\overline Q_T)]^4$,
 \bess
\lim_{n\to+\yy}\big(H_u(x,t+nT),\,H_i(x,t+nT),\, V_u(x,t+nT),\, V_i(x,t+nT)\big)=\big({\mathsf H}(x,t), 0, 0, 0\big).\eess

For the case of $\mathcal{R}_{0,1}, \mathcal{R}_{0,2} \le 1 $, we have that, in $[C^{2,1}(\overline Q_T)]^4$,
  \bess
  \lim_{n\to+\yy}\big(H_u(x,t+nT),\,H_i(x,t+nT),\, V_u(x,t+nT),\, V_i(x,t+nT)\big)=(0, 0, 0, 0).\eess
 \end{enumerate}\vspace{-2mm}
  \end{theo}

\begin{proof} By use of the regularity theory and compactness argument, it suffices to show that these limits hold uniformly in $\overline Q_T$.

(1)\, Assume that $\mathcal{R}_{0,1}, \mathcal{R}_{0,2}>1$ and $\mathcal{R}_0>1$.

Let $\varphi_j$ be the positive eigenfunction corresponding to $\zeta_j$ of $(2.3_j)$ with normalizing $\varphi_j$ by $\|\varphi_j\|_{L^{\infty}(Q_T)}=1$.

{\it Step 1}. Prove that for any given $\ep>0$ small enough, there exists a $N_\ep\gg 1$ such that
\bes\left\{\begin{aligned}
	&0<{\mathsf H}(x,t)-\ep\varphi_1(x,t)\le H(x,t+nT)\le {\mathsf H}(x,t)+\ep\varphi_1(x,t),&&(x,t)\in Q_T,\; n\geq N_\ep,\\[1mm]
	&0<{\mathsf V}(x,t)-\ep\varphi_2(x,t)\le V(x,t+nT)\le {\mathsf V}(x,t)+\ep\varphi_2(x,t),&&(x,t)\in Q_T,\; n\geq N_\ep.
\end{aligned}\rr.\label{equ:main:step1}\ees

This proof is motivated by \cite{Wang24}.
The condition $\mathcal{R}_0>1$ is equivalent to $\lm({\mathsf H}, {\mathsf V})<0$. So, there exists a $0<\ep_0\ll 1$ such that $\lm({\mathsf H}, {\mathsf V}; \ep)<0$ when $|\ep|\le\ep_0$. Let $(H_u, H_i, V_u, V_i)$ be the unique solution of \qq{1.1}. Then $H=H_u+H_i$ and $V=V_u+V_i$ satisfy, respectively, \qq{3.1} and \qq{3.2} with $T_{\rm max}=+\yy$. Moreover,
 \bes
 \lim_{n\to+\yy}H(x,t+nT)={\mathsf H}(x,t),\;\;\;\lim_{n\to+\yy}V(x,t+nT)={\mathsf V}(x,t) \;\;\;{\rm in}\;\;C^{2,1}(\overline Q_T).
 \lbl{3.6}\ees

For the case $\alpha_1=1$, we have $\min_{\overline Q_T}\varphi_1>0$. Thus,
 \bess
 |H(x,t+nT)-{\mathsf H}(x,t)|<\ep\varphi_1(x,t),\;\;\;(x,t)\in Q_T\eess
when $n$ is large, which yields \eqref{equ:main:step1}.

For the case $\alpha_1=0$, we have $\varphi_1>0$ in $Q_T$ and
 \[ \varphi_1(x,t)=0,\;\;\; \frac{\partial\varphi_1}{\partial\nu}(x,t)<0,\;\;\; H(x,t+nT)-{\mathsf H}(x,t)=0,\;\;\;(x,t)\in\overline S_T. \]
Noticing that $\lim_{n\to+\yy}(H(x,t+nT)-{{\mathsf H}}(x,t)\big)=0$ in $C^{2,1}(\overline Q_T)$, and $\frac{\partial\varphi_1}{\partial\nu}$ is continuous and negative on $\overline S_T$, we can find $N_1\gg 1$ such that
  \[\frac{\partial(H(x,t+nT)-{\mathsf H}(x,t)\big)}{\partial\nu}
  -\ep\frac{\partial\varphi_1}{\partial\nu}(x,t)\ge-\frac\ep 2\min_{\overline S_T}\frac{\partial\varphi_1}{\partial\nu}(x,t)>0,\;\;\;(x,t)\in\overline S_T\]
for all $n\ge N_1$. This together with $(H(x,t+nT)-{\mathsf H}(x,t)-\ep\varphi_1(x,t)\big)\big|_{\overline S_T}=0$ asserts that there exists a sub-domain $\Omega_0\Subset\Omega$ such that
$H(x,t+nT)-{\mathsf H}(x,t)-\ep\varphi_1(x,t)\le 0$ in $(\Omega\setminus\Omega_0)\times[0,T]$ for all $n\ge N_1$. Since $\ep\varphi_1(x,t)>0$ in $\overline\Omega_0\times[0,T]$ and $H(x,t+nT)-{\mathsf H}(x,t)\to 0$ in $C(\overline\Omega_0\times[0,T])$ as $n\to+\yy$, there exists $N_\ep>N_1$ such that $H(x,t+nT)-{\mathsf H}(x,t)-\ep\varphi_1(x,t)\le 0$ in $\overline\Omega_0\times[0,T]$ for all $n\ge N_\ep$. It concludes that, for all $n\ge N_\ep$,
 \bess H(x,t+nT)\le{\mathsf H}(x,t)+\ep\varphi_1(x,t),\;\;\;(x,t)\in Q_T.\vspace{-2mm}\eess

The other inequalities in \qq{equ:main:step1} can be proved by the same way.

{\it Step 2}. Prove the following limiting behavior:
 \bes
\limsup_{n\to+\yy}\big(H_i(x,t+nT), V_i(x,t+nT)\big)\leq \big({\mathsf H}_i(x,t), {\mathsf V}_i(x,t)\big) \;\;\;{\rm uniformly\; in}\; \;\overline Q_T.
 \lbl{equ:main:step2}\ees
Since ${\mathsf H}, {\mathsf V}, \varphi_2, \varphi_2$ and all coefficient functions  are $T$-periodic in time $t$. Making use of $H_u=H-H_i$, $V_u=V-V_i$ and \qq{equ:main:step1}, we see that $(H_i, V_i)$ satisfies
 \bess\left\{\begin{aligned}
&\partial_t H_i\le\nabla\cdot d_1(x,t)\nabla H_i+ \ell_1(x,t)\big({\mathsf H}(x,t)+\ep\varphi_1(x,t)-H_i\big)^+V_i\nm\\
 &\hspace{16mm}-\big[b_1(x,t)+c_1(x,t)\big({\mathsf H}(x,t)-\ep\varphi_1(x,t)\big)
 \big]H_i,&&(x,t)\in\oo\times(N_\ep T,\yy),\\
&\partial_t V_i\le\nabla\cdot d_2(x,t)\nabla V_i+ \ell_2(x,t)\big({\mathsf V}(x,t)+\ep\varphi_2(x,t)-V_i\big)^+H_i\nm\\
 &\hspace{16mm}-\big[b_2(x,t)+c_2(x,t)\big({\mathsf V}(x,t)-\ep\varphi_2(x,t)\big)
 \big]V_i,&&(x,t)\in\oo\times(N_\ep T,\yy),\\[1mm]
&\alpha_1\dd\frac{\partial H_i}{\partial\nu}+\beta_1(x,t)H_i=\alpha_2\dd\frac{\partial V_i}{\partial\nu}+\beta_2(x,t)V_i=0,&&(x,t)\in\partial\oo\times(N_\ep T,\yy).
 \end{aligned}\rr.\eess
Let $(\widehat{{\mathsf H}}_i^\ep, \widehat{{\mathsf V}}_i^\ep)$ be the unique positive $T$-periodic solution of \qq{2.7}, and $(\phi^\ep_1,\phi^\ep_2)$ be the positive eigenfunction corresponding to $\lm({\mathsf H}, {\mathsf V}; \ep)$. We can take constants $k\gg 1$ and $0<\delta\ll 1$ such that
 \bess
 k\big(\widehat{{\mathsf H}}_i^\ep(x,0), \dd\widehat{{\mathsf V}}_i^\ep(x,0)\big)\ge\big(H_i(x, N_\ep T), V_i(x, N_\ep T)\big)\ge\delta\big(\phi^\ep_1(x,0),\dd\phi^\ep_2(x,0)\big),\;\;\; x\in\oo.
 \eess
Moreover, it is easy to verify that $k(\widehat{{\mathsf H}}_i^\ep, \widehat{{\mathsf V}}_i^\ep)$ and $\delta(\phi^\ep_1,\phi^\ep_2)$ are the ordered upper and lower periodic solutions of \qq{2.7} provided $k\gg 1$ and $0<\delta\ll 1$. Let $(H_i^{\ep}, V_i^{\ep})$ be the unique positive solution of
 \bes\left\{\begin{aligned}
&\partial_t H_i^{\ep}=\nabla\cdot d_1(x,t)\nabla H_i^{\ep}
+\ell_1(x,t)\big({\mathsf H}(x,t)+\ep\varphi_1(x,t)-H_i^{\ep}\big)^+V_i^{\ep}\\
 &\hspace{16mm}-\big[b_1(x,t)+c_1(x,t)\big({\mathsf H}(x,t)-\ep\varphi_1(x,t)\big)
 \big]H_i^{\ep}, &&x\in\oo,\; t>0,\\
&\partial_t V_i^{\ep}=\nabla\cdot d_2(x,t)\nabla V_i^{\ep}
+\ell_2(x,t)\big({\mathsf V}(x,t)+\ep\varphi_2(x,t)-V_i^{\ep}\big)^+H_i^{\ep}\\
 &\hspace{16mm}-\big[b_2(x,t)+c_2(x,t)\big({\mathsf V}(x,t)-
 \ep\varphi_2(x,t)\big)\big]V_i^{\ep},&&x\in\oo,\; t>0,\\[1mm]
&\alpha_1\dd\frac{\partial H_i^{\ep}}{\partial\nu}+\beta_1(x,t)H_i^{\ep}=\alpha_2\dd\frac{\partial V_i^{\ep}}{\partial\nu}+\beta_2(x,t)V_i^{\ep}=0,&&x\in\partial\oo,\; t>0,\\
&(H_i^{\ep}(x,0), V_i^{\ep}(x,0)\big)=k\big(\widehat{{\mathsf H}}_i^\ep(x,0), \widehat{{\mathsf V}}_i^\ep(x,0)\big),&&x\in\Omega.
 \end{aligned}\rr.\label{3.8}\ees
By the comparison principle, $(H_i^{\ep}, V_i^{\ep})\ge\delta(\phi^\ep_1,\phi^\ep_2)$ and
 \bes
\big(H_i(x,t+N_\ep T),V_i(x,t+N_\ep T)\big)\leq\big(H_i^{\ep}(x,t), V_i^{\ep}(x,t)\big)\leq\big(k\widehat{{\mathsf H}}_i^\ep(x,t),k\widehat{{\mathsf V}}_i^\ep(x,t)\big)
 \lbl{3.9}\ees
for $x\in\oo$ and $t>0$. Define\vspace{-3mm}
 \[H_{i,n}^{\ep}(x,t)=H_i^{\ep}(x,t+nT),\;\;\;V_{i,n}^{\ep}(x,t)=V_i^{\ep}(x,t+nT).\]
Then
 \bes
 H_{i,n}^{\ep}(x,t)\ge\delta\phi^\ep_1(x,t+nT)=\delta\phi^\ep_1(x,t),\;\;\;
  V_{i,n}^{\ep}(x,t)\ge\delta\phi^\ep_2(x,t+nT)=\delta\phi^\ep_2(x,t)
  \lbl{3.10}\ees
for $(x,t)\in Q_T$. On the other hand, since $d_i, \ell_i,  b_i, c_i, \varphi_i, \beta_i$ and ${\mathsf H}, {\mathsf V}$ are time periodic functions with periodic $T$, we have
 \bes\left\{\begin{aligned}
&\partial_t H_{i,n}^{\ep}=\nabla\cdot d_1(x,t)\nabla H_{i,n}^{\ep}+\ell_1(x,t)\big({\mathsf H}(x,t)+\ep\varphi_1(x,t)-H_{i,n}^{\ep}\big)^+V_{i,n}^{\ep}\\
&\hspace{16mm}-\big[b_1(x,t)+c_1(x,t)\big({\mathsf H}(x,t)
 -\ep\varphi_1(x,t)\big)\big]H_{i,n}^{\ep},&&(x,t)\in Q_T,\\
&\partial_t V_{i,n}^{\ep}=\nabla\cdot d_2(x,t)\nabla V_{i,n}^{\ep}+\ell_2(x,t)\big({\mathsf V}(x,t)+\ep\varphi_2(x,t)-V_{i,n}^{\ep}\big)^+H_{i,n}^{\ep}\\
&\hspace{17mm}-\kk[b_2(x,t)+c_2(x,t)\big({\mathsf V}(x,t)-\ep\varphi_2(x,t)
 \big)\rr]V_{i,n}^{\ep},&&(x,t)\in Q_T,\\[1mm]
 &\alpha_1\dd\frac{\partial H_{i,n}^{\ep}}{\partial\nu}+\beta_1(x,t)H_{i,n}^{\ep}
 =\alpha_2\dd\frac{\partial V_{i,n}^{\ep}}{\partial\nu}+\beta_2(x,t)V_{i,n}^{\ep}=0,
 &&(x,t)\in S_T,\\
&H_{i,n}^{\ep}(x,0)=H_{i,n-1}^{\ep}(x,T), \ \ V_{i,n}^{\ep}(x,0)=V_{i,n-1}^{\ep}(x,T),&&x\in\Omega.
 \end{aligned}\rr.\lbl{3.11}\ees
Note that
\bess
&H_{i,1}^{\ep}(x,0)=H_i^{\ep}(x,T)\leq k\widehat{{\mathsf H}}_i^\ep(x,T)
=k\widehat{{\mathsf H}}_i^\ep(x,0)=H_i^{\ep}(x,0),\;\;\;x\in\Omega,\\
&V_{i,1}^{\ep}(x,0)=V_i^{\ep}(x,T)\leq k\widehat{{\mathsf V}}_i^\ep(x,T)
=k\widehat{{\mathsf V}}_i^\ep(x,0)=V_i^{\ep}(x,0),\;\;\;x\in\Omega.
\eess
By the comparison principle,
 \[\big(H_{i,1}^{\ep}(x,t), V_{i,1}^{\ep}(x,t)\big)\leq\big(H_i^{\ep}(x,t), V_i^{\ep}(x,t)\big),\;\;\;(x,t)\in Q_T,\]
and then
  \[H_{i,2}^{\ep}(x,0)=H_{i,1}^{\ep}(x,T)\leq H_i^{\ep}(x,T)=H_{i,1}^{\ep}(x,0),\;\;\; V_{i,2}^{\ep}(x,0)=V_{i,1}^{\ep}(x,T)\leq V_i^{\ep}(x,T)=V_{i,1}^{\ep}(x,0)\]
in $\Omega$. It is derived that
 \[\big(H_{i,2}^{\ep}(x,t), V_{i,2}^{\ep}(x,t)\big)\leq\big(H_{i,1}^{\ep}(x,t),
  V_{i,1}^{\ep}(x,t)\big),\;\;\;(x,t)\in Q_T\]
by the comparison principle. Utilizing the inductive method, we can show that $H_{i,n}^{\ep}$ and $V_{i,n}^{\ep}$ are monotonically decreasing in $n$.  This combined with \qq{3.10} indicates that there exists a function pair $({\mathring{\mathsf H}_i^{\ep}}, {\mathring{\mathsf V}_i^{\ep}})$ satisfying
 \[({\mathring{\mathsf H}_i^{\ep}}, {\mathring{\mathsf V}_i^{\ep}})\ge\delta
 (\phi^\ep_1,\dd\phi^\ep_2)\;\;\;\mbox{in}\;\;\in Q_T\]
such that $\lim_{n\to+\yy}(H_{i,n}^{\ep}, V_{i,n}^{\ep})=({\mathring{\mathsf H}_i^{\ep}}, {\mathring{\mathsf V}_i^{\ep}})$ pointwisely in $\overline Q_T$. Clearly, ${\mathring{\mathsf H}_i^{\ep}}(\cdot,0)={\mathring{\mathsf H}_i^{\ep}}(\cdot, T)$ and ${\mathring{\mathsf V}_i^{\ep}}(\cdot,0)={\mathring{\mathsf V}_i^{\ep}}(\cdot, T)$. By use of the regularity theory and compactness argument, it can be proved that $\lim_{n\to+\yy}(H_{i,n}^{\ep}, V_{i,n}^{\ep})=({\mathring{\mathsf H}_i^{\ep}}, {\mathring{\mathsf V}_i^{\ep}})$ in $[C^{2,1}(\overline Q_T)]^2$, and $({\mathring{\mathsf H}_i^{\ep}}, {\mathring{\mathsf V}_i^{\ep}})$ is a positive solution of \qq{2.7}. Thus,
 \[\big({\mathring{\mathsf H}_i^{\ep}}(x,t), {\mathring{\mathsf V}_i^{\ep}}(x,t)\big)=\big(\widehat{{\mathsf H}}_i^\ep(x,t), \widehat{{\mathsf V}}_i^\ep(x,t)\big)\]
by the uniqueness of positive solution of \qq{2.7}. Consequently,
 \[\lim_{n\to+\yy}H_i^{\ep}(x,t+nT)=\widehat{{\mathsf H}}_i^\ep(x,t),\;\;\;
\lim_{n\to+\yy}V_i^{\ep}(x,t+nT)=\widehat{{\mathsf V}}_i^\ep(x,t) \;\;\; {\rm in}\;\; C^{2,1}(\overline Q_T).\]
Note that $\lim_{\ep\to0}\big({\widehat{{\mathsf H}}_i^\ep}(x,t), {\widehat{{\mathsf V}}_i^\ep}(x,t)\big)=({{\mathsf H}_i}(x,t), {{\mathsf V}_i}(x,t)\big)$ in $[C^{2,1}(\overline Q_T)]^2$. Combining with  \eqref{3.9}, we can obtain \eqref{equ:main:step2}.

{\it Step 3}. Prove that for $\tau >0$ small enough, there exists $N_\tau$ such that the following two estimates hold:
\bes
H_i(x,t+nT)&<&{\mathsf H}(x,t)-\tau\varphi_1(x,t), \;\;(x,t)\in\overline Q_T,\label{equ:main:step3.1}\\[1mm]
V_i(x,t+nT)&<&{\mathsf V}(x,t)-\tau\varphi_2(x,t), \;\;(x,t)\in\overline Q_T
\label{equ:main:step3.2} \ees
for all $n\geq N_\tau$.

Let $\widetilde{\mathsf H}={\mathsf H}-{\mathsf H}_i$ and $\widetilde{\mathsf V}={\mathsf V}-{\mathsf V}_i$. Then $\widetilde{\mathsf H}>0$ and $\widetilde{\mathsf V}>0$ in $Q_T$ (Theorem \ref{th2.1}). The direct calculation yields that $\widetilde{\mathsf H}$ and $\widetilde{\mathsf V}$ satisfy
 \bess\kk\{\begin{aligned}
 &\partial_t \widetilde{\mathsf H}-\nabla\cdot d_1(x,t)\nabla\widetilde{\mathsf H}+\big[b_1(x,t)+c_1(x,t){\mathsf H}(x,t)
 +\ell_1(x,t){\mathsf V}_i(x,t)\big]\widetilde{\mathsf H}\\
 &\hspace{20mm}=a_1(x,t){\mathsf H}(x,t),&&(x,t)\in Q_T,\\
 &\alpha_1\dd\frac{\partial\widetilde{\mathsf H}}{\partial\nu}+\beta_1(x,t)\widetilde{\mathsf H}=0,&&(x,t)\in S_T,\\
 &\widetilde{\mathsf H}(x,0)=\widetilde{\mathsf H}(x,T),&&x\in\oo.
 \end{aligned}\rr.\eess
and
 \bess\kk\{\begin{aligned}
&\partial_t \widetilde{\mathsf V}-\nabla\cdot d_2(x,t)\nabla\widetilde{\mathsf V}+\big[b_2(x,t)+c_2(x,t){\mathsf V}(x,t)
 +\ell_2(x,t){\mathsf H}_i(x,t)\big]\widetilde{\mathsf V}\\
 &\hspace{20mm}=a_2(x,t){\mathsf V}(x,t),&&(x,t)\in Q_T,\\
 &\alpha_2\dd\frac{\partial\widetilde{\mathsf V}}{\partial\nu}+\beta_2(x,t)\widetilde{\mathsf V}=0,
 &&(x,t)\in S_T,\\
 &\widetilde{\mathsf V}(x,0)=\widetilde{\mathsf V}(x,T),&&x\in\oo,
 \end{aligned}\rr.\eess
respectively. Recalling that
 \bess\begin{aligned}
 &b_1(x,t)+c_1(x,t){\mathsf H}(x,t)
 +\ell_1(x,t){\mathsf V}_i(x,t)\ge 0,\,\not\equiv 0,&&(x,t)\in Q_T,\\
 &b_2(x,t)+c_2(x,t){\mathsf V}(x,t)+\ell_2(x,t){\mathsf H}_i(x,t)\ge 0,\,\not\equiv 0,&&(x,t)\in Q_T,\end{aligned}\eess
and
  \bess
  a_1(x,t){\mathsf H}(x,t)>0,\;\;a_2(x,t){\mathsf V}(x,t)>0,\;\;\;(x,t)\in Q_T.
  \eess
We can show that there exists $0<\tau_0<\ep_0$ such that $\widetilde{\mathsf H}>2\tau\varphi_1$, $\widetilde{\mathsf V}>2\tau\varphi_2$, i.e.,
   \bes
{\mathsf H}_i<{\mathsf H}-2\tau\varphi_1,\; \;\; {\mathsf V}_i<{\mathsf V}-2\tau\varphi_2\;\;\;\mbox{in}\;\; Q_T
   \lbl{3.14}\ees
for all $0<\tau<\tau_0$. In fact, when $a_1=1$, then $\widetilde{\mathsf H}>0$ in $\overline Q_T$, and so $\widetilde{\mathsf H}>2\tau\varphi_1$ in $\overline Q_T$ if $\tau>0$ is suitable small. When $a_1=0$, then $\widetilde{\mathsf H}>0$ in $Q_T$, and $\widetilde{\mathsf H}=0$ and $\frac{\partial\widetilde{\mathsf H}}{\partial\nu}<0$ on $S_T$ by the maximum principle (\cite[Theorem 7.1]{Wpara}) and Hopf boundary lemma. By adapting analogous methods as in \cite[Lemma 2.1]{WPbook24} we can show that $\widetilde{\mathsf H}>2\tau\varphi_1$ in $Q_T$ if $\tau>0$ is suitable small. Similarly, $\widetilde{\mathsf V}>2\tau\varphi_2$ holds in $Q_T$ if $\tau>0$ is small.

For such $0<\tau<\tau_0$, since $\lim_{\ep\to0}{\mathring{\mathsf H}_i^{\ep}}
=\lim_{\ep\to0}\widehat{{\mathsf H}}_i^\ep={\mathsf H}_i$ in $C^{2,1}(\overline Q_T)$, by the analogous discussions of \qq{equ:main:step1}, we can find $0<\bar\ep<\ep_0$ such that
 \[{\mathring{\mathsf H}_i^{\bar \ep}}<{\mathsf H}_i+(\tau/2)\vp_1\;\;\;\mbox{in}\;\; Q_T.\]
Recalling that $\lim_{n\to+\yy}H_i^{\bar \ep}(x,t+nT)={\mathring{\mathsf H}_i^{\bar \ep}}(x,t)$ in $C^{2,1}(\overline Q_T)$, we can find $N_{\bar\ep, \tau}\gg 1$ such that
 \bes
 H_i^{\bar \ep}(x,t+nT)<{\mathring{\mathsf H}_i^{\bar \ep}}(x,t)+(\tau/2)\vp_1(x,t)<{\mathsf H}_i(x,t)+\tau \vp_1(x,t),\;\;\;(x,t)\in\overline Q_T,\;\;\forall\; n\geq N_{\bar\ep,\tau}.
 \quad\lbl{3.15}\ees
By \qq{3.9}, $H_i(x,t+(n+N_{\bar\ep, \tau})T)\leq H_i^{\bar \ep}(x,t+nT)$ holds in $Q_T$. It follows from \qq{3.14} and \qq{3.15} that
$$
 H_i(x,t+(n+N_{\bar\ep,\tau})T)<{\mathsf H}_i(x,t)+\tau \vp_1(x,t)<{\mathsf H}(x,t)-\tau\varphi_1(x,t), \;\;(x,t)\in\overline Q_T
$$
for all $n\geq N_{\bar\ep, \tau}$. The desired  \eqref{equ:main:step3.1} can be obtained by chosen $N_\tau= N_{\bar\ep, \tau}$.

 Similarly, we can prove \eqref{equ:main:step3.2}.

{\it Step 4}. Completes the proof.

For such $\tau$ determined in Step 3. Thanks to \qq{equ:main:step1}, there exists $\hat N_\tau\gg 1$ such that
 \bess\left\{\begin{aligned}
&0<{\mathsf H}(x,t)-\tau\varphi_1(x,t)\le H(x,t+nT)\le {\mathsf H}(x,t)+\tau\varphi_1(x,t),&&(x,t)\in Q_T,\; n\geq \hat N_\tau,\\
&0<{\mathsf V}(x,t)-\tau\varphi_2(x,t)\le V(x,t+nT)\le {\mathsf V}(x,t)+\tau\varphi_2(x,t),&&(x,t)\in Q_T,\; n\geq \hat N_\tau.
  \end{aligned}\rr.\eess
Choose $N_\tau$ such that \eqref{equ:main:step3.1} and \eqref{equ:main:step3.2} hold. Set $N_*=N_\tau+\hat N_\tau$. Noticing that $H_u=H-H_i$,  $V_u=V-V_i$, and ${\mathsf H}, {\mathsf V}, \varphi_1, \varphi_2$ are $T$-periodic in time $t$. Take advantage of \qq{equ:main:step3.1} and \qq{equ:main:step3.2}, it follows that $(H_i, V_i)$ satisfies
 \bess\kk\{\begin{aligned}
&\partial_t H_i\ge\nabla\cdot d_1(x,t)\nabla H_i+ \ell_1(x,t)\big({\mathsf H}(x,t)-\tau\varphi_1(x,t)-H_i\big)^+V_i\\
 &\hspace{16mm}-\big[b_1(x,t)+c_1(x,t)\big({\mathsf H}(x,t)+\tau\varphi_1(x,t)\big)\big]H_i\; ,&&(x,t)\in\oo\times(N_*T,\yy),\\
&\partial_t V_i\ge\nabla\cdot d_2(x,t)\nabla V_i+ \ell_2(x,t)\big({\mathsf V}(x,t)-\tau\varphi_2(x,t)-V_i\big)^+H_i\\
&\hspace{16mm}-\big[b_2(x,t)+c_2(x,t)\big({\mathsf V}(x,t)+\tau\varphi_2(x,t)\big)\big]V_i\; ,&&(x,t)\in\oo\times(N_*T,\yy),\\[1mm]
&\alpha_1\dd\frac{\partial H_i}{\partial\nu}+\beta_1(x,t)H_i=\alpha_2\dd\frac{\partial V_i}{\partial\nu}+\beta_2(x,t)V_i=0\; &&(x,t)\in\partial\oo\times(N_*T,\yy).
 \end{aligned}\rr.\eess

Let $(\phi^{-\tau}_1,\phi^{-\tau}_2\big)$ be the positive eigenfunction corresponding to $\lm({\mathsf H}, {\mathsf V}; -\tau)$. Similar to Step 1, we can find $0<\delta\ll 1$ such that $\delta(\phi^{-\tau}_1,\phi^{-\tau}_2\big)$ is a lower solution of \qq{2.7} with $\ep=-\tau$, and $\delta(\phi^{-\tau}_1(x,0),\phi^{-\tau}_2(x,0)\big)\le \big(H_i(x, N_*T), V_i(x, N_*T)\big)$ in $\oo$. Let $(H_i^{-\tau}, V_i^{-\tau})$ be the unique positive solution of
 \bess\kk\{\begin{aligned}
&\partial_t H_i^{-\tau}=\nabla\cdot d_1(x,t)\nabla H_i^{-\tau}+\ell_1(x,t)\big({\mathsf H}(x,t)
-\tau\varphi_1(x,t)-H_i^{-\tau}\big)^+V_i^{-\tau}\\
 &\hspace{16mm}-\big[b_1(x,t)+c_1(x,t)\big({\mathsf H}(x,t)
 +\tau\varphi_1(x,t)\big)\big]H_i^{-\tau},&&x\in\oo,\;t>0,\\
&\partial_t V_i^{-\tau}=\nabla\cdot d_2(x,t)\nabla V_i^{-\tau}=\ell_2(x,t)\big({\mathsf V}(x,t)
-\tau\varphi_2(x,t)-V_i^{-\tau}\big)^+H_i^{-\tau}\\
 &\hspace{16mm}-\big[b_2(x,t)+c_2(x,t)\big({\mathsf V}(x,t)
 +\tau\varphi_2(x,t)\big)\big]V_i^{-\tau},&&x\in\oo,\;t>0,\\[1mm]
&\alpha_1\dd\frac{\partial H_i^{-\tau}}{\partial\nu}+\beta_1(x,t)H_i^{-\tau}=\alpha_2\dd\frac{\partial V_i^{-\tau}}{\partial\nu}+\beta_2(x,t)V_i^{-\tau}=0,&&x\in\partial\oo,\;t>0,\\
&(H_i^{-\tau}(x,0), V_i^{-\tau}(x,0)\big)=
\delta(\phi^{-\tau}_1(x,0),\phi^{-\tau}_2(x,0)\big),&&x\in\oo.
\end{aligned}\rr.\eess
By the comparison principle, for $x\in\oo$ and $t>0$,
 \bes
\big(H_i(x,t+N_*T), V_i(x,t+N_*T)\big)\geq (H_i^{-\tau}(x,t), V_i^{-\tau}(x,t)\big)\geq (\delta \phi^{-\tau}_1(x,t),\delta \phi^{-\tau}_2(x,t)\big).
\lbl{3.16}\ees
Set
 \[H_{i,n}^{-\tau}(x,t)=H_i^{-\tau}(x,t+nT),\;\;\;V_{i,n}^{-\tau}(x,t)
 =V_i^{-\tau}(x,t+nT)\]
for $(x,t)\in \overline Q_T$. Then $(H_{i,n}^{-\tau}, V_{i,n}^{-\tau})$ satisfies \qq{3.11} with $\ep=-\tau$. Similar to Step 2, we can prove that $H_{i,n}^{-\tau}$ and $V_{i,n}^{-\tau}$ are monotonically increasing in $n$, and
 \[\lim_{n\to+\yy}H_i^{-\tau}(x,t+nT)={\mathsf H}_i^{-\tau}(x,t), \ \ \
\lim_{n\to+\yy}V_i^{-\tau}(x,t+nT)={\mathsf H}_i^{-\tau}(x,t) \ \ {\rm in}\ C^{2,1}(\overline Q_T),\]
where $({\mathsf H}_i^{-\tau}, {\mathsf V}_i^{-\tau})$ is the unique positive periodic solution of problem \qq{2.7} with $\ep=-\tau$. Take advantage of  \eqref{3.16} and $\lim_{\tau\to 0}({{\mathsf H}_i^{-\tau}}, {{\mathsf V}_i^{-\tau}})=({{\mathsf H}_i}, {{\mathsf V}_i})$, it follows that
\[\liminf_{n\to+\yy}\big(H_i(x,t+nT), V_i(x,t+nT)\big)\geq \big({\mathsf H}_i(x,t), {\mathsf V}_i(x,t)\big) \;\;\;{\rm uniformly\; in}\; \;\overline Q_T.\]
This, together with \eqref{equ:main:step2}, yields that
 \[\lim_{n\to+\yy}\big(H_i(x,t+nT),V_i(x,t+nT)\big)=\big({\mathsf H}_i(x,t),{\mathsf V}_i(x,t)\big) \;\;\;{\rm uniformly\; in}\; \;\overline Q_T.\]

Using the fact that $\lim_{n\to+\yy}(V_u(x,t+nT)+V_i(x,t+nT)\big)={\mathsf V}(x,t)$ in $C^{2,1}(\overline Q_T)$ and the uniform estimate (cf. \cite[Theorems 2.11 and 3.14]{Wpara}), we see that \qq{3.3} holds. The conclusion (1) is proved.

(2)\, Assume that $\mathcal{R}_{0,1}, \mathcal{R}_{0,2}>1$ and $\mathcal{R}_0 \le 1$. Let $\lm({\mathsf H}, {\mathsf V}; \ep)$ be the principal eigenvalue of \qq{2.7}. If $\lm({\mathsf H}, {\mathsf V})>0$, then $\lm({\mathsf H}, {\mathsf V}; \ep)>0$ when $0<\ep\ll 1$. It can be seen from the proof of necessity of Theorem \ref{th2.1} that \qq{2.8} has no positive solution. This implies that \qq{2.7} has no positive solution. In fact, if $(\widehat{{\mathsf H}}_i^\ep, \widehat{{\mathsf V}}_i^\ep)$ is a positive solution of \qq{2.7}, then $(\widehat{{\mathsf H}}_i^\ep, \widehat{{\mathsf V}}_i^\ep)<({\mathsf H}+\ep\varphi_1, {\mathsf V}+\ep\varphi_2)$ since $({\mathsf H}+\ep\varphi_1, {\mathsf V}+\ep\varphi_2)$ is a strict upper solution of \qq{2.7}. Therefore, $(\widehat{{\mathsf H}}_i^\ep, \widehat{{\mathsf V}}_i^\ep)$ is also a positive solution of \qq{2.8}, which is impossible.

Let $(H_i^{\ep},V_i^{\ep})$ be the unique positive solution of \qq{3.8} with initial data $(C, C)$, where $C$ is a suitably large positive constant, for example,
  \bess
 C>\max_{\overline\oo\times[0,+\yy)}\big({\mathsf H}(x,t)+{\mathsf V}(x,t)+H_i(x,t)+V_i(x,t)+\ep\vp_1(x,t)+\ep\vp_2(x,t)\big),\eess
such that $(C, C)$ is an upper solution of \qq{2.7}. Similar to the above,
 \[\lim_{n\to+\yy}\big(H_i^{\ep}(x,t+nT),V_i^{\ep}(x,t+nT)\big)=(0,0)\]
in $[C^{2,1}(\overline Q_T)]^2$ since \qq{2.7} has no positive solution. By the comparison principle we have
 \[\lim_{t\to+\yy}\big(H_i(x,t+nT), V_i(x,t+nT)\big)=(0,0).\]
This, together with
 \bess
 \lim_{n\to+\yy}\big(V_u(x,t+nT)+V_i(x,t+nT)\big)={\mathsf V}(x,t),\\[0.5mm]
 \lim_{n\to+\yy}\big(H_u(x,t+nT)+H_i(x,t+nT)\big)={\mathsf H}(x,t),\eess
yields the limit \qq{3.4}.\vskip 2pt

If $\lm({\mathsf V})=0$, then $\lm({\mathsf V}; \ep)<0$ and \qq{2.7} has a unique positive solution $({\mathsf H}_i^\ep, {\mathsf V}_i^\ep)$ for any $\ep>0$. Obviously, $\lim_{\ep\to 0}({\mathsf H}_i^\ep, {\mathsf V}_i^\ep)=(0, 0)$ in $[C^{2,1}(\overline Q_T)]^2$. Similar to the above, $\lim_{n\to+\yy}\big(H_i(x,t+nT), V_i(x,t+nT))=(0,0)$ and \qq{3.4} holds.

The proof of conclusion (2) is complete.

(3)\; Assume that $\mathcal{R}_{0,1} \le 1$ and $\mathcal{R}_{0,2}>1$. Then $\lim_{n\to+\yy}V(x,t+nT)={\mathsf V}(x,t)$ and $(2.2_1)$ has no positive solution. Therefore,  $\lim_{n\to+\yy}H(x,t+nT)=0$, and then
  \[\lim_{n\to+\yy}H_u(x,t+nT)=\lim_{n\to+\yy}H_i(x,t+nT)\big)=0\;\;\;{\rm uniformly \; in }\;\;\overline Q_T.\]
Consequently,
 \[\lim_{n\to+\yy}V_i(x,t+nT)=0, \;\;\;{\rm and\; then}\;\; \lim_{n\to+\yy}V_u(x,t+nT)={\mathsf V}(x,t)\;\;\;{\rm uniformly \;in}\;\; \overline Q_T.\]

The rest conclusions can be proven in a similar way. The proof is complete.
 \end{proof}

\section{Special cases: autonomous case and constant case}

In this section, we first present the dynamics of \qq{1.1} for the autonomous case, that is, the coefficient functions do not depend on time $t$. Then we employ the explicit expression to demonstrate the dynamics of \qq{1.1} subject to the Neumann type boundary condition for the constant case.
For the autonomous case, the time periodic problem \qq{2.1} associated to \qq{1.1} becomes equilibrium problem
  \bes\kk\{\begin{aligned}
&-\nabla\cdot d_1(x)\nabla{\mathsf H}_u=
a_1(x)({\mathsf H}_u+{\mathsf H}_i)-b_1(x){\mathsf H}_u\\
&\hspace{32mm}-c_1(x)({\mathsf H}_u+{\mathsf H}_i){\mathsf H}_u
-\ell_1(x){\mathsf H}_u{\mathsf V}_i,&&x\in\oo,\\
&-\nabla\cdot d_1(x)\nabla{\mathsf H}_i=\ell_1(x){\mathsf H}_u{\mathsf V}_i
-b_1(x){\mathsf H}_i-c_1(x)({\mathsf H}_u
+{\mathsf H}_i){\mathsf H}_i,&&x\in\oo,\\
&-\nabla\cdot d_2(x)\nabla{\mathsf V}_u=
a_2(x)({\mathsf V}_u+{\mathsf V}_i)-b_2(x){\mathsf V}_u\\
&\hspace{32mm}-c_2(x)({\mathsf V}_u+{\mathsf V}_i){\mathsf V}_u
-\ell_2(x){\mathsf H}_i{\mathsf V}_u,&&x\in\oo,\\
&-\nabla\cdot d_2(x)\nabla{\mathsf V}_i=\ell_2(x){\mathsf H}_i{\mathsf V}_u-b_2(x) {\mathsf V}_i-c_2(x)({\mathsf V}_u+{\mathsf V}_i){\mathsf V}_i,&&x\in\oo,\\
&\alpha_1\dd\frac{\partial {\mathsf H}_u}{\partial\nu}+\beta_1(x){\mathsf H}_u
=\alpha_1\frac{\partial{\mathsf H}_i}{\partial\nu}
+\beta_1(x){\mathsf H}_i=0,&&x\in\partial\oo,\\[1mm]
&\alpha_2\dd\frac{\partial{\mathsf V}_u}{\partial\nu}+\beta_2(x){\mathsf V}_u
=\alpha_2\frac{\partial{\mathsf V}_i}{\partial\nu}
+\beta_2(x){\mathsf V}_i=0,&&x\in\partial\oo.
  \end{aligned}\rr.\label{4.1}\ees
We first consider the problems
\[\left\{\begin{array}{lll}
	-\nabla\cdot d_j(x)\nabla{\mathsf U}=\big(a_j(x)-b_j(x)\big){\mathsf U}
	-c_j(x){\mathsf U}^2\;\;,\;&x\in\oo,\\[1.5mm]
	\alpha_j(x)\dd\frac{\partial{\mathsf U}}{\partial\nu}+\beta_j(x){\mathsf U}=0&{\rm on}\;\; \partial\oo
\end{array}\right.\eqno(4.2_j).\]
with $j=1,2$. Then ${\mathsf H}={\mathsf H}_u+{\mathsf H}_i$ and ${\mathsf V}={\mathsf V}_u+{\mathsf V}_i$ are the positive solutions of $(4.2_1)$ and $(4.2_2)$, respectively.

In the meantime, $(2.3_j)$ and \qq{2.4} become
\[\begin{cases}
	-\nabla\cdot d_j(x)\nabla\varphi+\big(b_j(x)-a_j(x)\big)\varphi=
	\zeta\varphi,\;&x\in\oo,\\[1.5mm]	 \alpha_j(x)\dd\frac{\partial\varphi}{\partial\nu}+\beta_j(x)\varphi=0,&x\in\partial\oo,
\end{cases}\]
and\setcounter{equation}{2}
  \bes\kk\{\begin{aligned}
&-\nabla\cdot d_1(x)\nabla{\mathsf H}_i=\ell_1(x)\big({\mathsf H}(x)
-{\mathsf H}_i){\mathsf V}_i-\big[b_1(x)+c_1(x){\mathsf H}(x)\big]
{\mathsf H}_i,&&x\in\oo,\\
&-\nabla\cdot d_2(x)\nabla{\mathsf V}_i=\ell_2(x)\big({\mathsf V}(x)
-{\mathsf V}_i){\mathsf H}_i-\big[b_2(x)+c_2(x){\mathsf V}(x)\big]
{\mathsf V}_i,&&x\in\oo,\\[.1mm]
&\alpha_1\dd\frac{\partial{\mathsf H}_i}{\partial\nu}+\beta_1(x){\mathsf H}_i
=\alpha_2\dd\frac{\partial{\mathsf V}_i}{\partial\nu}+\beta_2(x){\mathsf V}_i=0,&&x\in\partial\oo,
 \end{aligned}\rr.\label{4.3}\ees
respectively, and the eigenvalue problem \qq{2.5} becomes
  \bes\kk\{\begin{aligned}
&-\nabla\cdot d_1(x)\nabla\phi_1+\big[b_1(x)+c_1(x){\mathsf H}(x)\big]\phi_1
= \ell_1(x){\mathsf H}(x)\phi_2+\lm\phi_1,&&x\in\oo,\\
&-\nabla\cdot d_2(x)\nabla\phi_2+\big[b_2(x)+c_2(x){\mathsf V}(x)\big]\phi_2
=\ell_2(x){\mathsf V}(x)\phi_1+\lm\phi_2,&&x\in\oo,\\[.1mm]
&\alpha_1\dd\frac{\partial\phi_1}{\partial\nu}+\beta_1(x)\phi_1	 =\alpha_2\dd\frac{\partial\phi_2}{\partial\nu}+\beta_2(x)\phi_2=0,&&x\in\partial\oo.
  \end{aligned}\rr.\label{4.4}\ees
The principal eigenvalue of \qq{4.4} is still denoted by $\lm({\mathsf H}, {\mathsf V})$. Moreover, the basic reproduction ratios $\mathcal{R}_{0,j}$, $\mathcal{R}_0$ can be defined by \cite{WZ2012SIADS}. We remark that $\Phi_j(t)$ and $\Phi(t)$ are the semiflow on $Y_j$ and $X$, respectively, as well as, $F_j$ and $F$ are independent of $t$. Therefore, the operators $L_j$ and $\hat{L}_j$ are defined on $Y_j$, and $\mathcal{L}$ and $\hat{\mathcal{L}}$ are defined on $X$. Here we omit the details. The corresponding results are the following theorems.\vspace{-2mm}

\begin{theo}\lbl{th4.2} Assume that $\mathcal{R}_{0,j}>1$, $j=1,2$. Then problem \eqref{4.3} has a positive solution $\big({\mathsf H}_i, {\mathsf V}_i\big)$ if and only if $\mathcal{R}_0>1$. Moreover, the positive solution $\big({\mathsf H}_i, {\mathsf V}_i\big)$ of \eqref{4.3} is unique and satisfies ${\mathsf H}_i<{\mathsf H}$ and ${\mathsf V}_i<{\mathsf V}$ when it exists. Therefore, \eqref{4.1}  has a positive solution $\big({\mathsf H}_u,\, {\mathsf H}_i,\,{\mathsf V}_u,\,{\mathsf V}_i\big)$ if and only if $\mathcal{R}_0>1$, and such solution is unique and takes the form
	\[({\mathsf H}_u, {\mathsf H}_i,{\mathsf V}_u,{\mathsf V}_i)=({\mathsf H}-{\mathsf H}_i,{\mathsf H}_i,{\mathsf V}-{\mathsf V}_i,{\mathsf V}_i)\]
	when it exists.\vspace{-2mm}
\end{theo}

\begin{theo}\lbl{th4.1a} Assume that the coefficient functions do not depend on time $t$ in \qq{1.1}, and let $(H_u, H_i, V_u, V_i)$ be the unique positive solution of \qq{1.1}.\vspace{-2mm}
	\begin{enumerate}[$(1)$]
		\item\,If $\mathcal{R}_{0,1},\mathcal{R}_{0,2}>1$ and $\mathcal{R}_0>1$, then
		\[\lim_{t\to+\yy}\big(H_u(x,t),\, H_i(x,t),\, V_u(x,t),\, V_i(x,t)\big)
		=\big({\mathsf H}(x)-{\mathsf H}_i(x),\, {\mathsf H}_i(x), \,{\mathsf V}(x)-{\mathsf V}_i(x), \, {\mathsf V}_i(x)\big) \]
		in $[C^{2,1}(\overline\oo)]^4$, where ${\mathsf H}(x)$, ${\mathsf V}(x)$ and $\big({\mathsf H}_i(x), {\mathsf V}_i(x)\big)$ are the unique positive solutions of $(4.2_1)$, $(4.2_2)$ and \qq{4.3}, respectively, and ${\mathsf H}_i(x)<{\mathsf H}(x), {\mathsf V}_i(x)<{\mathsf V}(x)$.
		\item\,If $\mathcal{R}_{0,1}, \mathcal{R}_{0,2}>1$ and $\mathcal{R}_0 \le 1$, then
		\[\lim_{t\to+\yy}\big(H_u(x,t),\,H_i(x,t),\, V_u(x,t),\, V_i(x,t)\big)=\big({\mathsf H}(x), 0, {\mathsf V}(x), 0\big)\;\;\;
		{\rm in}\;\; [C^{2,1}(\overline\oo)]^4.\]
		\item\,If $\mathcal{R}_{0,1} \le 1$ and $\mathcal{R}_{0,2}>1$, then
		\[\lim_{t\to+\yy}\big(H_u(x,t),\,H_i(x,t),\, V_u(x,t),\, V_i(x,t)\big)=(0, 0, {\mathsf V}(x), 0)\;\;\;
		{\rm in}\;\; [C^{2,1}(\overline\oo)]^4.\]
		
		If $\mathcal{R}_{0,2} \le 1$ and $\mathcal{R}_{0,1}>1$, then
		\[\lim_{t\to+\yy}(H_u(x,t),\,H_i(x,t),\, V_u(x,t),\, V_i(x,t)\big)=\big({\mathsf H}(x), 0, 0, 0)\;\;\;
		{\rm in}\;\; [C^{2,1}(\overline\oo)]^4.\]
		
		If $\mathcal{R}_{0,1},\mathcal{R}_{0,2}\le 1$, then
		\[\lim_{t\to+\yy}\big(H_u(x,t),\,H_i(x,t),\, V_u(x,t),\, V_i(x,t)\big)=(0, 0, 0, 0)\;\;\;{\rm in}\;\; [C^{2,1}(\overline\oo)]^4. \]
	\end{enumerate}\vspace{-2mm}
\end{theo}

We then further consider the constant case subject to the Neumann type boundary conditions. The time periodic problem \qq{2.1} associated to \qq{1.1} becomes algebraic equations
	\bes\kk\{\begin{aligned}
	&a_1(H_u+H_i)-b_1H_u-c_1(H_u+H_i)H_u-\ell_1H_uV_i=0,\\
&\ell_1H_uV_i-b_1H_i-c_1(H_u+H_i)H_i=0,\\
&a_2(V_u+V_i)-b_2V_u-c_2(V_u+V_i)V_u-\ell_2H_iV_u=0,\\
	&\ell_2H_iV_u-b_2 V_i-c_2(V_u+V_i)V_i=0.
	\vspace{-1mm}\end{aligned}\rr.\label{4.5}\ees
The corresponding
 \bess
\zeta_1=b_1-a_1,\;\;\;
\zeta_2=b_2-a_2,\;\;\;
\mathcal{R}_{0,1}=\frac{a_1}{b_1},\;\;\;
\mathcal{R}_{0,2}=\frac{a_1}{b_1},\;\;\;
{\mathsf H}=\frac{a_1-b_1}{c_1},\;\;\;
{\mathsf V}=\frac{a_2-b_2}{c_2},
	\eess
and \qq{2.4} and \qq{2.5} become (necessarily, $a_1>b_1$ and $a_2>b_2$) algebraic equations
	\bes\kk\{\begin{aligned}
&\ell_1\big[(a_1-b_1)/c_1-{\mathsf H}_i\big]{\mathsf V}_i-a_1{\mathsf H}_i=0,\\
&\ell_2\big[(a_2-b_2)/c_2-{\mathsf V}_i\big]{\mathsf H}_i-a_2{\mathsf V}_i=0,
\vspace{-1mm}\end{aligned}\rr.\label{4.6}\ees
	and the eigenvalue problem of algebraic equations
	\bes\kk\{\begin{aligned}
&a_1\phi_1-(\ell_1(a_1-b_1)/{c_1})\phi_2=\lm\phi_1,\\
& a_2\phi_2-(\ell_2(a_2-b_2)/{c_2})\phi_1=\lm\phi_2,
 \vspace{-1mm}\end{aligned}\rr.\label{4.7}\ees
respectively. The principal eigenvalue of \qq{4.7} is
  \[\lm({\mathsf H}, {\mathsf V})=\frac 12\kk(a_1+a_2-\sqrt{(a_1-a_2)^2
+4\frac{\ell_1\ell_2(a_1-b_1)(a_2-b_2)}{c_1c_2}}\rr).\]
The basic reproduction number is
 \bess
 \mathcal{R}_0=\sqrt{\frac{\ell_1\ell_2(a_1-b_1)(a_2-b_2)}{c_1c_2a_1a_2}},
 \eess
which is solved by the following eigenvalue problem (see e.g., \cite[Theorem 3.2]{WZ2012SIADS} or \cite[Theorem 3.8]{LZZ2019R0}) :
  \bess\kk\{\begin{aligned}
&a_1\phi_1=\frac{1}{\lm({\mathsf H}, {\mathsf V})}
\frac{\ell_1(a_1-b_1)}{c_1}\phi_2,\\[1mm]
&a_2\phi_2=\frac{1}{\lm({\mathsf H}, {\mathsf V})}\frac{\ell_2(a_2-b_2}{c_2}\phi_1.
\end{aligned}\rr.\eess
If $\lm({\mathsf H}, {\mathsf V}) < 0$ (or $\mathcal{R}_0>1$), i.e.,
	\[ a_1a_2<\frac{\ell_1\ell_2(a_1-b_1)(a_2-b_2)}{c_1c_2},\]
	then the unique positive solution $\big({\mathsf H}_i, {\mathsf V}_i)$ of \qq{4.6} takes the form:
	\bess{\mathsf H}_i=\frac{\ell_1\ell_2(a_1-b_1)(a_2-b_2)-a_1a_2c_1c_2}
	{c_1\ell_2[a_1c_2+\ell_1(a_2-b_2)]},\\[2mm]
	{\mathsf V}_i=\frac{\ell_1\ell_2(a_1-b_1)(a_2-b_2)-a_1a_2c_1c_2}
	{c_2\ell_1[a_2c_1+\ell_2(a_1-b_1)]}.\eess
We have the following conclusions.\vspace{-3mm}

\begin{theo}\lbl{th4.3} Assume that the coefficient functions are constants and boundary conditions are Neumann type. Assume that $a_1>b_1$ and $a_2>b_2$. Then problem \eqref{4.5} has a positive solution $\big({\mathsf H}_u, {\mathsf H}_i,{\mathsf V}_u,{\mathsf V}_i)$ if and only if $a_1a_2<\frac{\ell_1\ell_2(a_1-b_1)(a_2-b_2)}{c_1c_2}$. In this case, the positive solution of \eqref{4.5} is uniquely given by
	\bess
	{\mathsf H}_u&=&\frac{a_1c_2[a_2c_1+\ell_2(a_1-b_1)]}
	{c_1\ell_2[a_1c_2+\ell_1(a_2-b_2)]},\\[1mm]
	{\mathsf H}_i&=&\frac{\ell_1\ell_2(a_1-b_1)(a_2-b_2)-a_1a_2c_1c_2}
	{c_1\ell_2[a_1c_2+\ell_1(a_2-b_2)]},\\[1mm]
	{\mathsf V}_u&=&\frac{a_2c_1[a_1c_2+\ell_1(a_2-b_2)]}
	{c_2\ell_1[a_2c_1+\ell_2(a_1-b_1)]},\\[1mm]
	{\mathsf V}_i&=&\frac{\ell_1\ell_2(a_1-b_1)(a_2-b_2)-a_1a_2c_1c_2}
	{c_2\ell_1[a_2c_1+\ell_2(a_1-b_1)]}.
	\eess
\end{theo}
	
We also have the following conclusions.
\begin{theo}\lbl{th4.4} Assume that the coefficient functions are constants and boundary conditions are Neumann type. Let $(H_u(x,t),\, H_i(x,t),\, V_u(x,t),\, V_i(x,t)\big)$ be the unique positive solution of \qq{1.1}.\vspace{-2mm}
 \begin{enumerate}[$(1)$]
\item\, If
\[a_1>b_1, \;\; a_2>b_2,\;\; a_1a_2<\frac{\ell_1\ell_2(a_1-b_1)(a_2-b_2)}{c_1c_2},\]
then
\bess
\lim_{t\to+\yy}\big(H_u(x,t),\, H_i(x,t),\, V_u(x,t),\, V_i(x,t)\big)
=\big({\mathsf H}_u, {\mathsf H}_i, {\mathsf V}_u, {\mathsf V}_i\big)\;\;\;{\rm in}\;\;[C^{2,1}(\overline\oo)]^4.
		\eess
\item\, If
\[a_1>b_1, \;\; a_2>b_2,\;\;a_1a_2\ge\frac{\ell_1\ell_2(a_1-b_1)(a_2-b_2)}{c_1c_2},\]
	then
	\bess
		\lim_{t\to+\yy}\big(H_u(x,t),\,H_i(x,t),\, V_u(x,t),\, V_i(x,t)\big)=\kk((a_1-b_1)/{c_1},\; 0,\; (a_2-b_2)/{c_2},\; 0\rr)\;\;\;{\rm in}\;\;[C^{2,1}(\overline\oo)]^4.
 \eess
\item\, If $a_1\le b_1$ and $a_2>b_2$, then
\bess
\lim_{t\to+\yy}\big(H_u(x,t),\,H_i(x,t),\, V_u(x,t),\, V_i(x,t)\big)=\kk(0,\; 0,\; (a_2-b_2)/{c_2},\; 0\rr)\;\;\;
{\rm in}\;\; [C^{2,1}(\overline\oo)]^4.\eess
		
If $a_2\le b_2$ and $a_1>b_1$, then
\bess
\lim_{n\to+\yy}(H_u(x,t),\,H_i(x,t),\, V_u(x,t),\, V_i(x,t)\big)=\kk((a_1-b_1)/{c_1},\; 0,\; 0,\; 0\rr)\;\;\;
{\rm in}\;\; [C^{2,1}(\overline\oo)]^4.
\eess
		
If $a_1\le b_1$ and $a_2\le b_2$, then
\bess
\lim_{t\to+\yy}\big(H_u(x,t),\,H_i(x,t),\, V_u(x,t),\, V_i(x,t)\big)=(0, 0, 0, 0)\;\;\;{\rm in}\;\; [C^{2,1}(\overline\oo)]^4.
		\eess
	\end{enumerate}
\end{theo}

\section{Numerical simulation}

In this section, we conduct numerical simulations to investigate the effects of spatial and temporal heterogeneity on the basic reproduction number $\mathcal{R}_0$. Specifically, we employ a backward difference for the time variable and a central difference for the spatial variable to perform our simulations. We set $\Omega = [0, 1]$ and $T = 1$. The parameters are chosen as follows:
	\bess\begin{aligned}
 &d_1(x,t)=0.1,\;\;d_2(x,t)=0.2,\;\;l_1(x,t)=3,&\\
 & l_2(x,t)=2,\;\;c_1(x,t)=c_2(x,t)=1,&\\
  &a_1(x,t)=2(1+p_1\cos(\pi x))(1+p_3\cos(2\pi t)), & \\
		&a_2(x,t)=3(1+p_2\cos(\pi x))(1+p_4\cos(2\pi t)),& \\
	&b_1(x,t)=1(1+q_1\cos(\pi x))(1+q_3\cos(2\pi t)),&  \\
		 &b_2(x,t)=2(1+q_2\cos(\pi x))(1+q_4\cos(2\pi t)).&
	\end{aligned}\eess	
	
In the case where $p_1=p_2=p_3=p_4=q_1=q_2=q_3=q_4=0$ (the birth rates and death rates of susceptible birds and mosquitoes are neglected), we obtain that $\mathcal{H}(x,t) \equiv 1$ and $\mathcal{V}(x,t) \equiv 1$, and then $\mathcal{R}_0=1$. In the case where $p_1=p_2=p_3=p_4=0.5$ and $q_1=q_2=q_3=q_4=0$ (the death rates of susceptible birds and mosquitoes are neglected), we can compute $\mathcal{R}_0 \approx1.14>1$ numerically.
\begin{figure}[H]
	\begin{minipage}[t]{0.47\textwidth}
		\centering
		\includegraphics[width=\textwidth]{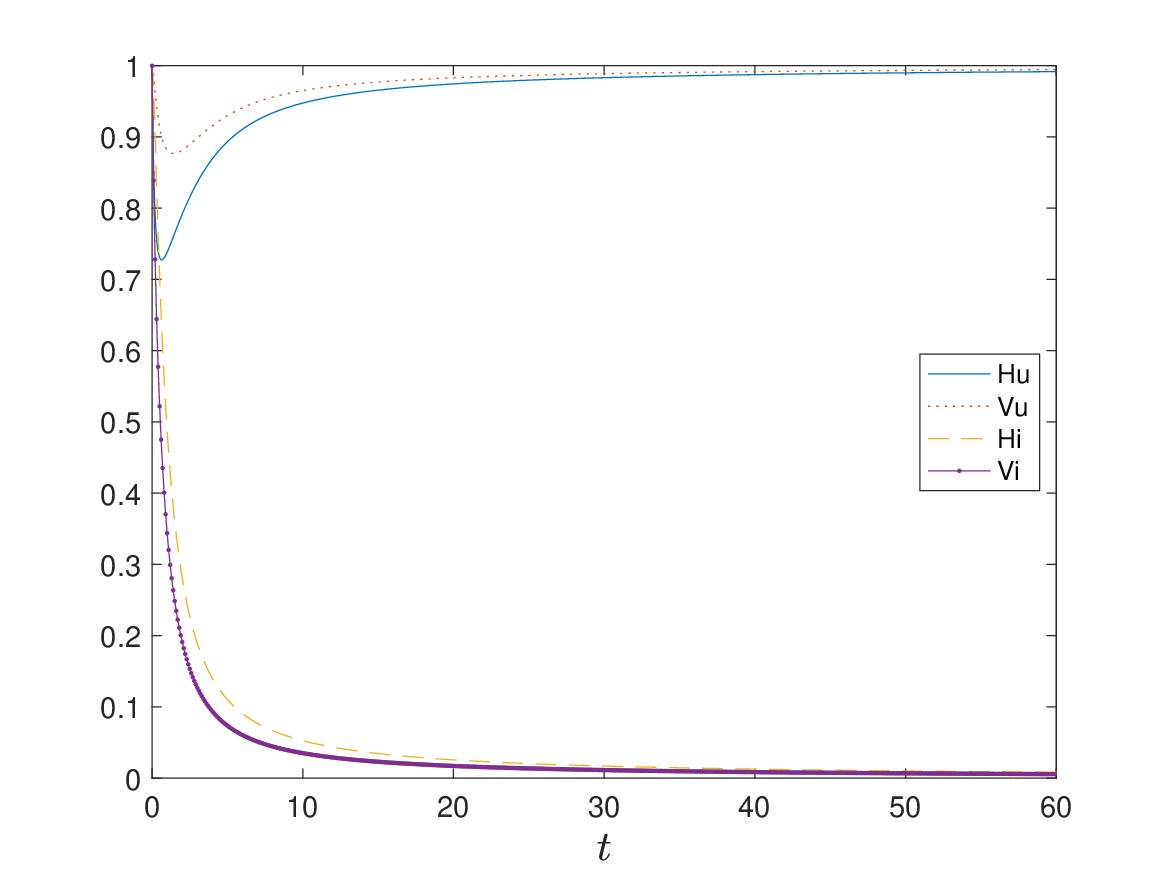}\\[-2mm]
		\caption*{\footnotesize (a)\, $p_1=p_2=p_3=p_4=q_1=q_2=q_3=q_4=0$}
	\end{minipage}
	\begin{minipage}[t]{0.47\textwidth}
		\centering
		\includegraphics[width=\textwidth]{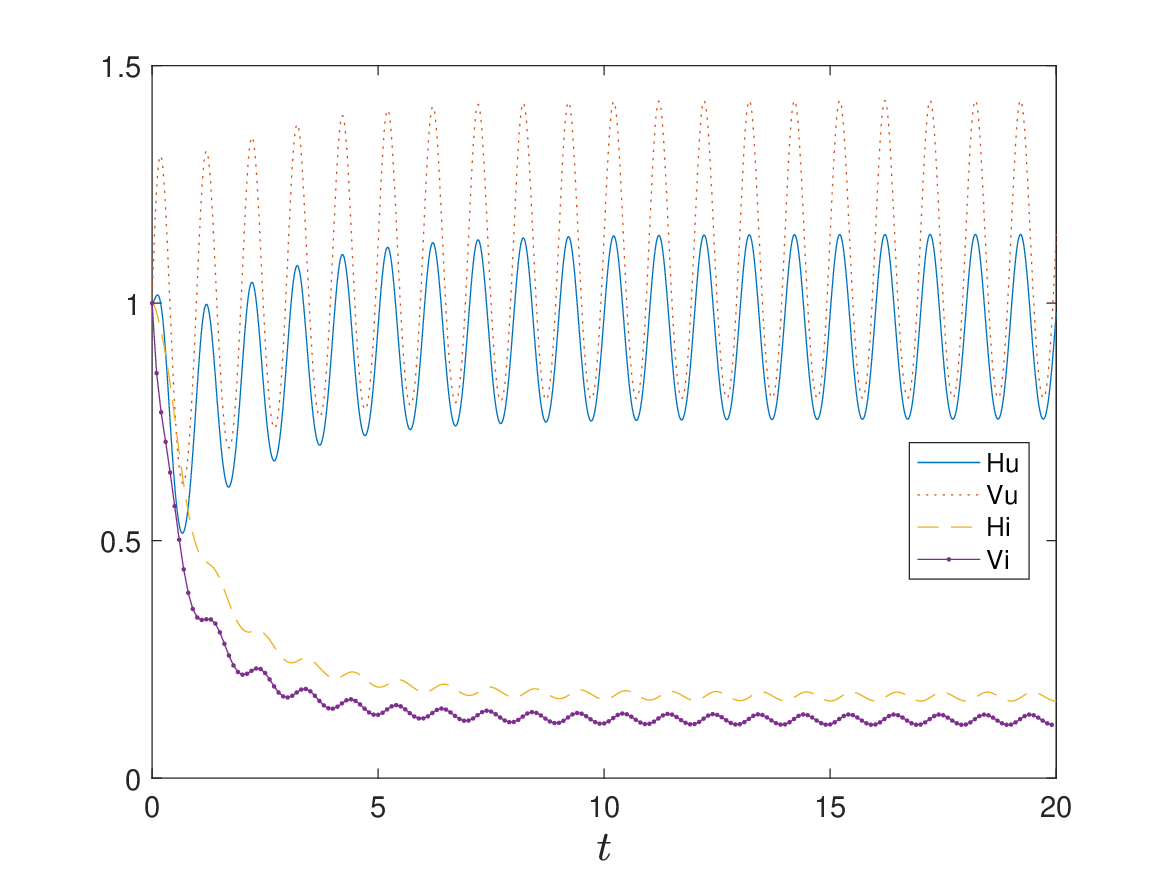}\\[-2mm]
	\caption*{\footnotesize (b)\, $p_1=p_2=p_3=p_4=0.5$, $q_1=q_2=q_3=q_4=0$}
	\end{minipage}\\[-2mm]
	\caption{\label{fig:solution}\small The spatial average of $H_u$, $H_i$, $V_u$ and $V_i$. In (a),  $H_i$ and $V_i$ will disappear eventually with $\mathcal{R}_0=1$. In (b),  $H_i$ and $V_i$ will persist eventually with $\mathcal{R}_0\approx1.14>1$.}\vspace{-3mm}
 \end{figure}
 \begin{figure}[htb]
\centering\begin{minipage}{0.4\textwidth}
\includegraphics[width=\textwidth]{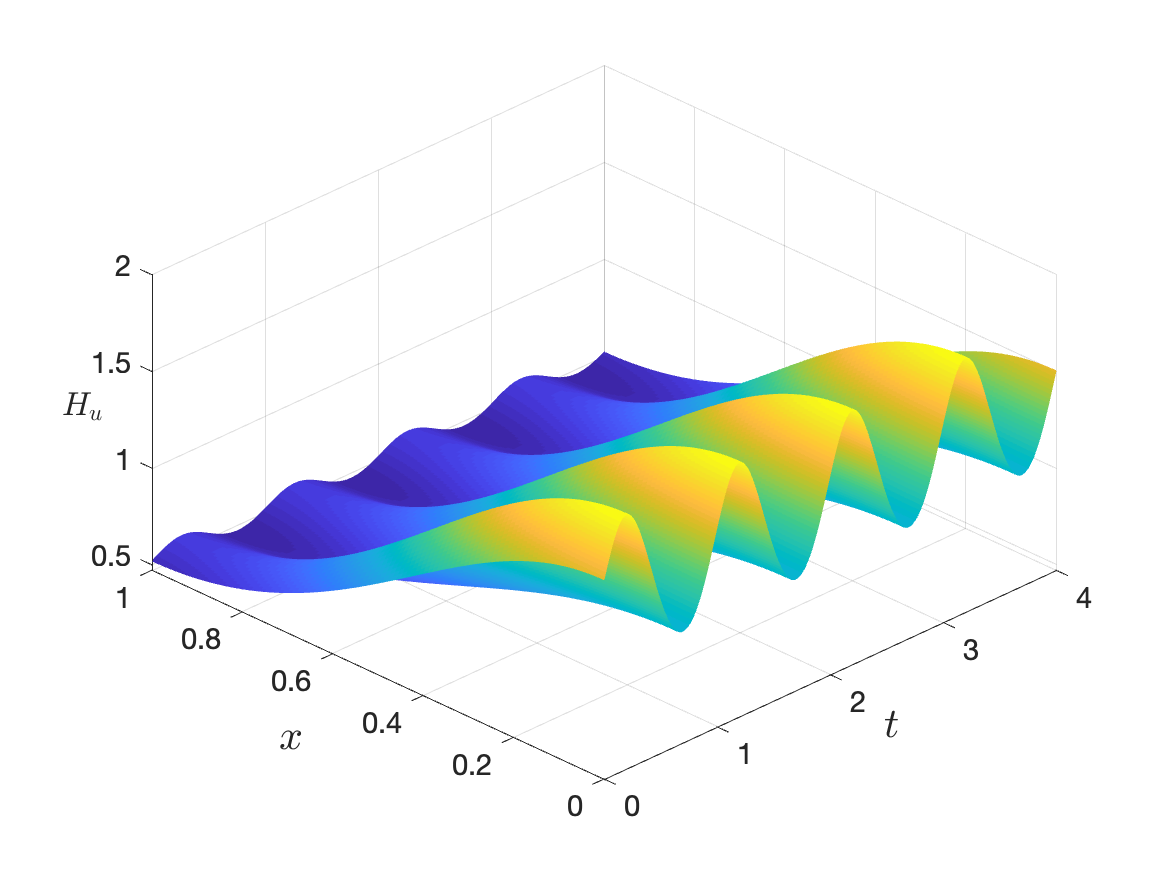}\end{minipage}
\begin{minipage}{0.4\textwidth}
\includegraphics[width=\textwidth]{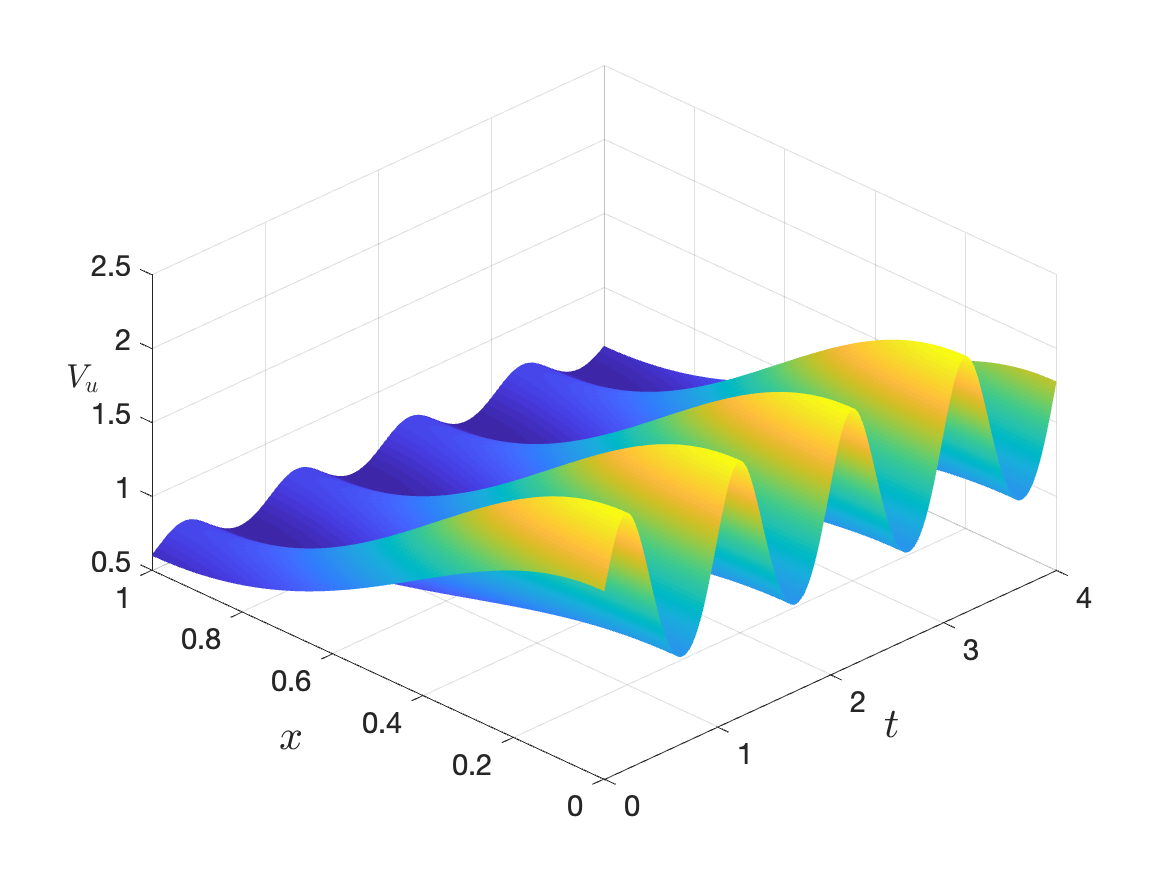}\end{minipage}\\
\centering\begin{minipage}{0.4\textwidth}
\includegraphics[width=\textwidth]{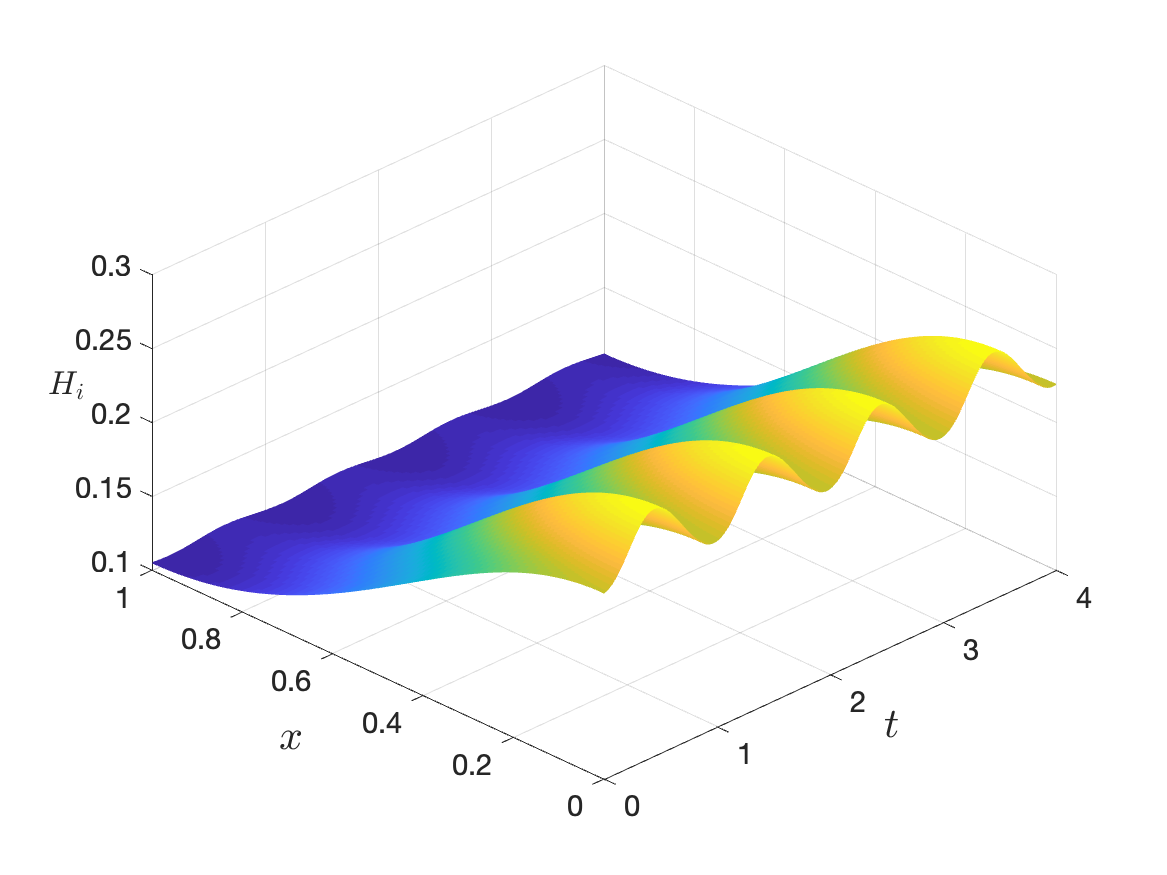}\end{minipage}%
	\begin{minipage}{0.4\textwidth}
\includegraphics[width=\textwidth]{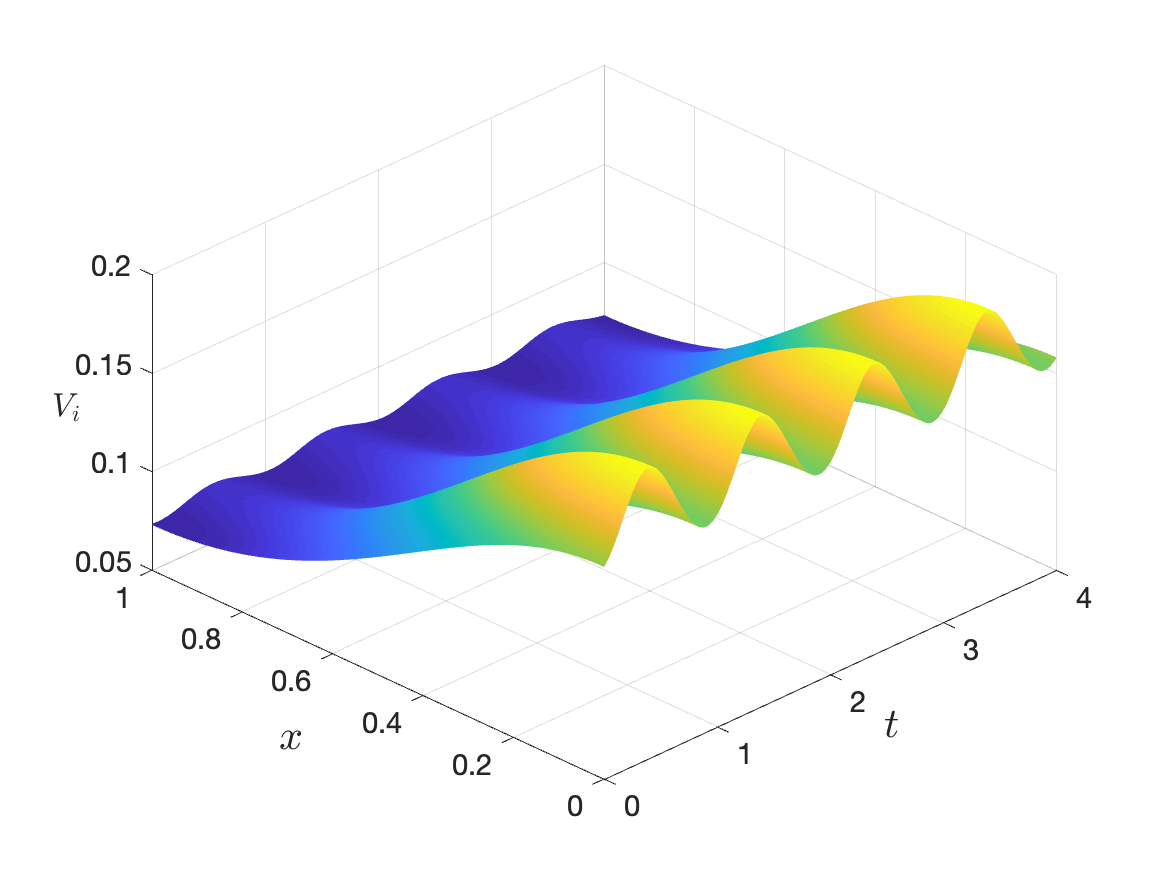}\end{minipage}\\[-2mm]
		\caption{\label{fig:PPS} \small Positive periodic solution when $p_1=p_2=p_3=p_4=0.5$ and $q_1=q_2=q_3=q_4=0$.}\vspace{-3mm}
	\end{figure}

A natural question arises: does greater temporal or spatial heterogeneity lead to a higher likelihood of disease propagation?

We first demonstrate the effect of temporal heterogeneity on $\mathcal{R}_0$ by showing  the level sets of $\mathcal{R}_0$ with respect to $p_3$ and $p_4$, as well as $q_3$ and $q_4$ by fixing $p_1=p_2=q_1=q_2=0$. It can be seen from Figures \ref{fig:R0:time}(a) and \ref{fig:R0:time}(b) that the temporal heterogeneity on $\mathcal{R}_0$ is very weak when we only consider the temporal heterogeneity on $a_1$ and $a_2$ or $b_1$ and $b_2$. However, temporal heterogeneity will have significant effect when we consider the temporal  heterogeneity on all of $a_1$, $a_2$, $b_1$ and $b_2$, as shown in Figures \ref{fig:R0:time}(c) and \ref{fig:R0:time}(d). In these two figures, $\mathcal{R}_0$ may be decreasing, increasing and non-monotone with respect to $p_3$ and $p_4$, which is very complicated.
  \begin{figure}[H]
  \centering
\subfigure[Level sets of $\mathcal{R}_0$ v.s. $p_3$, $p_4$ with  $q_3=q_4=0$]{\includegraphics[width=0.45\textwidth]{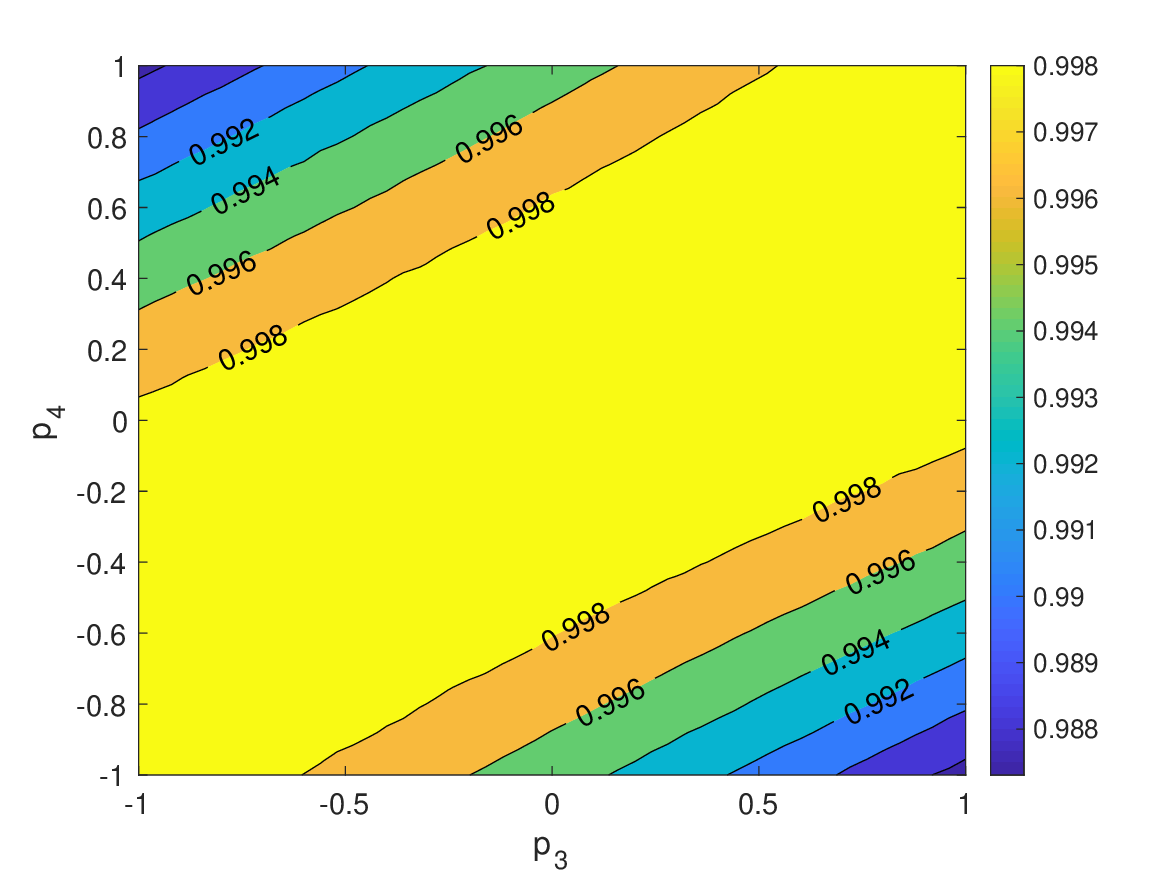}
\label{fig-sub3a}}\hspace{1mm}
	\subfigure[Level sets of $\mathcal{R}_0$ v.s. $q_3$, $q_4$ with $p_3=p_4=0$]
{\includegraphics[width=0.45\textwidth]{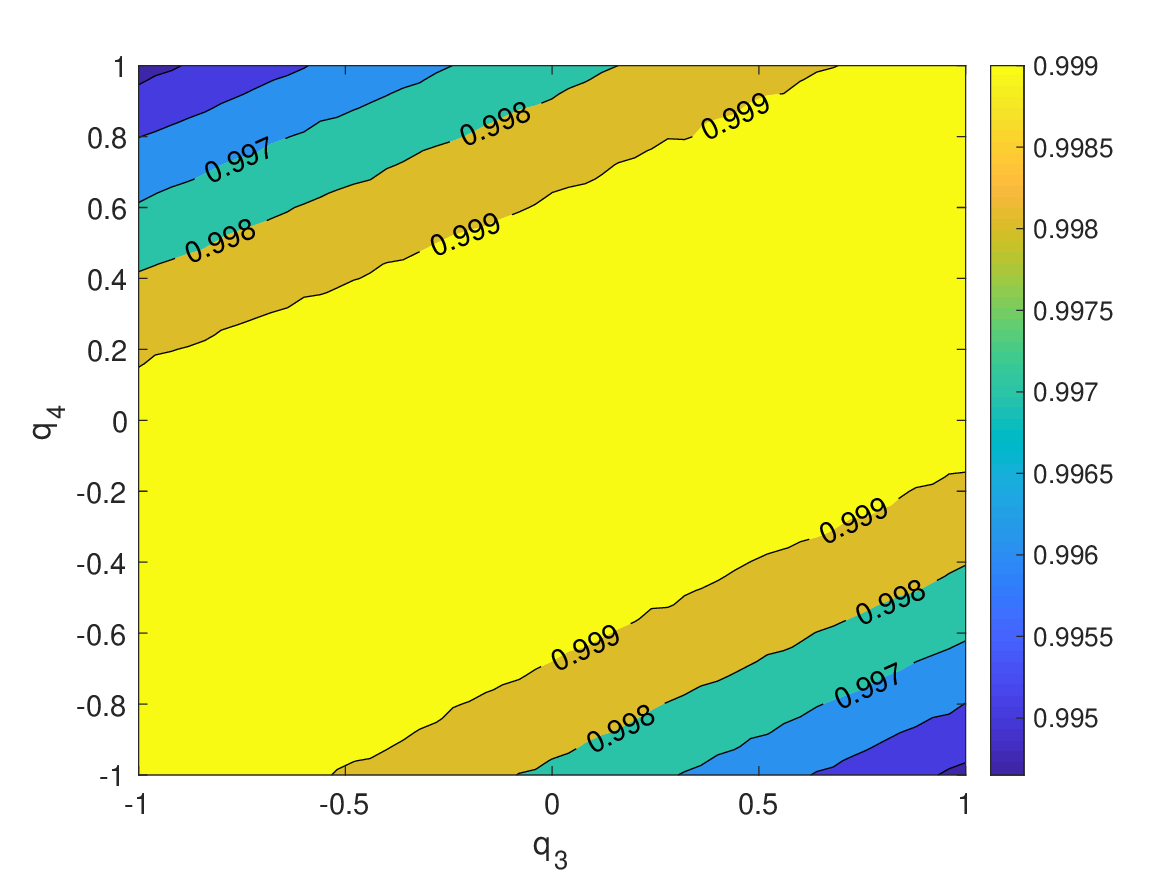}\label{fig-sub3b}}\\[2mm]
  \centering\subfigure[Level sets of $\mathcal{R}_0$ v.s. $p_3=q_3$, $p_4=q_4$]{
\includegraphics[width=0.45\textwidth]{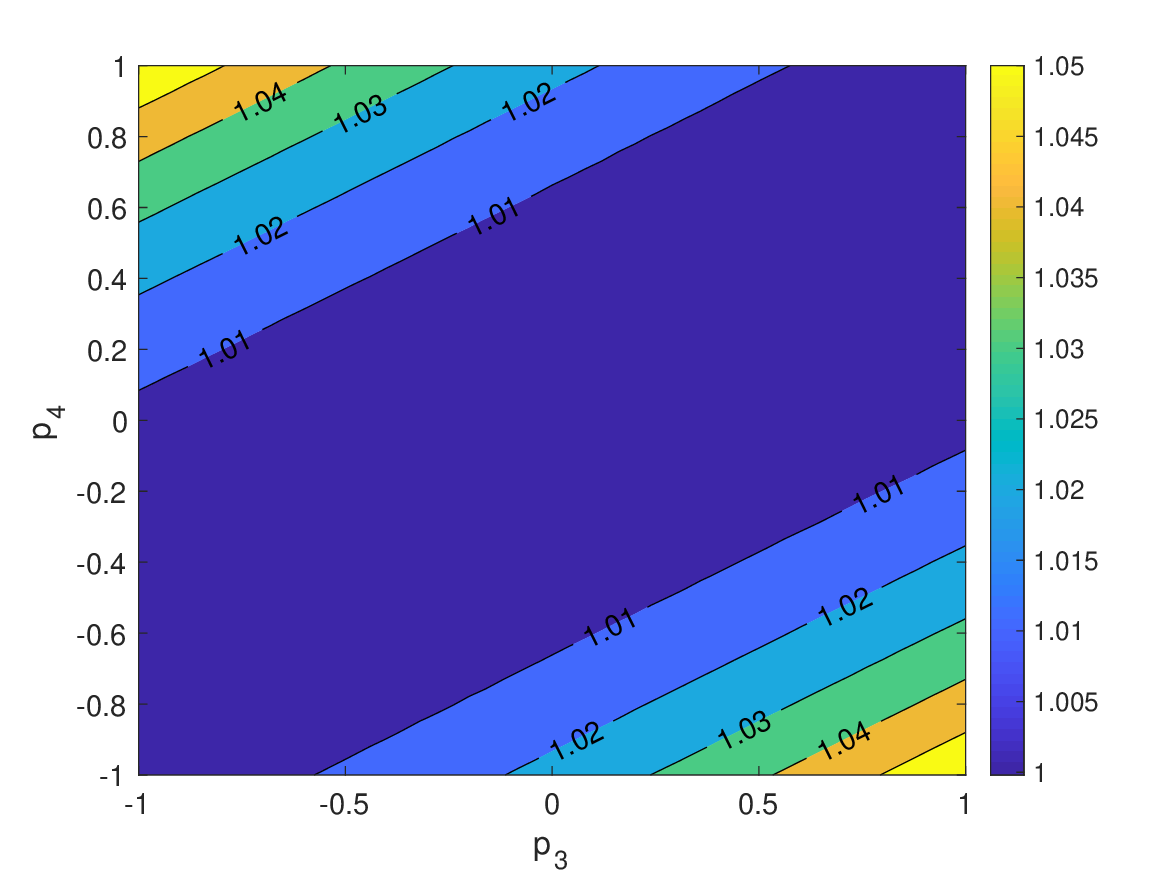}
\label{fig-sub3c}}\hspace{1mm}
	\subfigure[Level sets of $\mathcal{R}_0$ v.s. $p_3=-q_3$, $p_4=-q_4$]{
\includegraphics[width=0.45\textwidth]{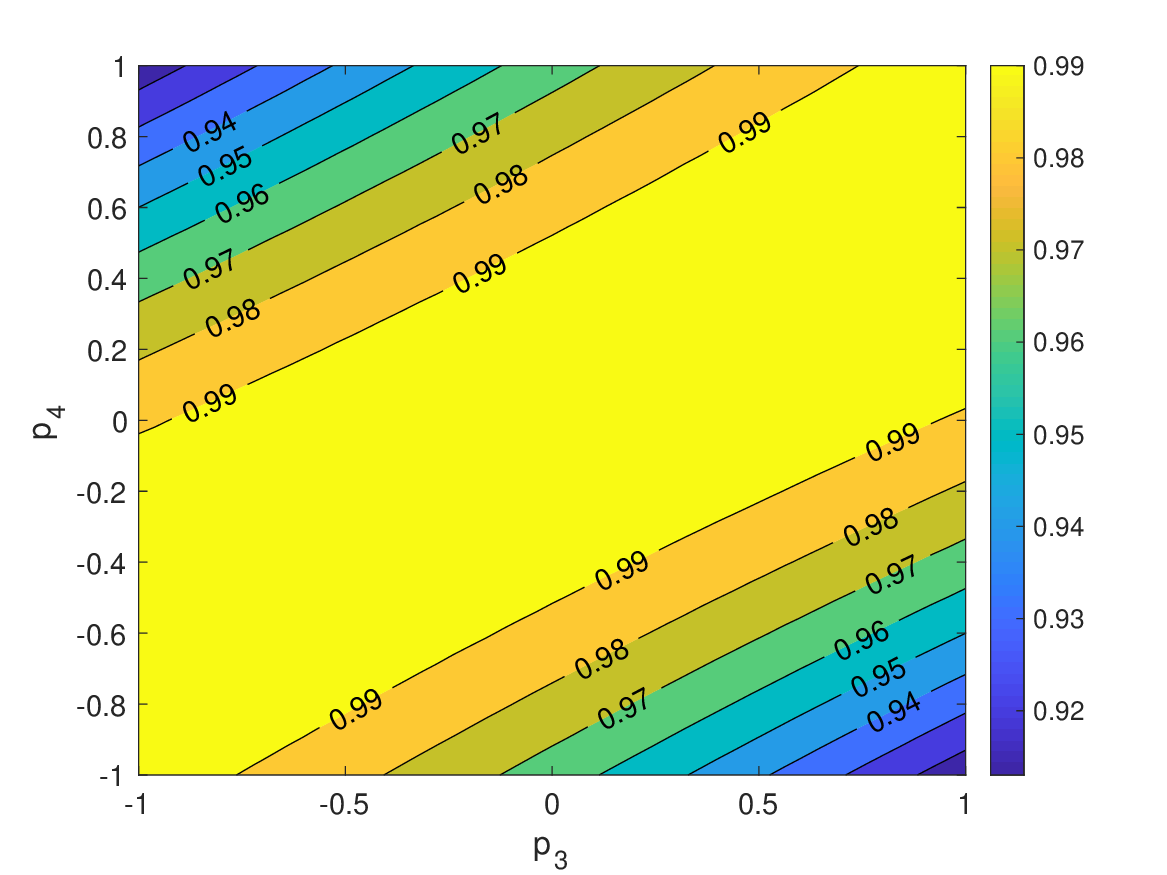}\label{fig-sub3d}}\\[-2mm]
\caption{\small Level sets of $\mathcal{R}_0$ with  $p_1=p_2=q_1=q_2=0$. The two figures in the first row show that the effect of $p_3$ and $p_4$, as well as, $q_3$ and $q_4$ on $\mathcal{R}_0$ are very weak. However, the superposition of weak effects may lead to significant influences, as shown in the two figures in the second row. \label{fig:R0:time}}\vspace{-3mm}
	\end{figure}

We next present the level set of $\mathcal{R}_0$ with respect to $p_1$ and $p_2$, as well as $q_1$ and $q_2$, which is to investigate the influence of spatial heterogeneity on $\mathcal{R}_0$ by fixing $p_3=p_4=q_3=q_4=0$. From Figure \ref{fig-sub3a}, we can see that $\mathcal{R}_0$ first decreases  and then increases with respect to $p_1$ (fixing $p_2$) and $p_2$ (fixing $p_1$), respectively. However, from Figure \ref{fig-sub3b}, $\mathcal{R}_0$ is  monotonically increasing, decreasing or first decreasing and then increasing with respect to $q_1$ when $q_2=0.8$, $q_2=0$, $q_2=-0.8$, respectively. Setting $p_1 = q_1$ and $p_2 = q_2$ may weaken the impact on $\mathcal{R}_0$ compared to changing only $p_1$ and $p_2$, or only $q_1$ and $q_2$ (see Figure \ref{fig-sub3c}). In contrast, setting $p_1 = -q_1$ and $p_2 = -q_2$ may enhance the impact on $\mathcal{R}_0$ compared to the same individual changes (see Figure \ref{fig-sub3d}). This implies that the effect of spatial heterogeneity on $\mathcal{R}_0$ is quite complicated, which depends on the parameters as well as the degree of spatial heterogeneity.

Finally, we focus on the impact of spatiotemporal factors in birth rates on $\mathcal{R}_0$ by taking death rates are constants, i.e., $q_1=q_2=q_3=q_4=0$.  Figure \ref{fig:R0:3} illustrates the influence of temporal and spatial heterogeneity on $\mathcal{R}_0$. It is easy to see that the change in $ p_1$ (or $p_2$) may have a greater impact on $\mathcal{R}_0$ than the change in $ p_3 $ (or $p_4$).  This implies that spatial heterogeneity may have greater influence on $\mathcal{R}_0$ than the temporal one in general. When we take birth rates as constants ($p_1=p_2=p_3=p_4=0$) and examine the impact of spatiotemporal factors in death rates on the basic reproduction number $\mathcal{R}_0$, we can see a similar phenomenon. Omit the simulation graphics here.

\begin{figure}[H]
		\centering
	\subfigure[Level sets of $\mathcal{R}_0$ v.s. $p_1$, $p_2$ with $q_1=q_2=0$]{\includegraphics[width=0.45\textwidth]{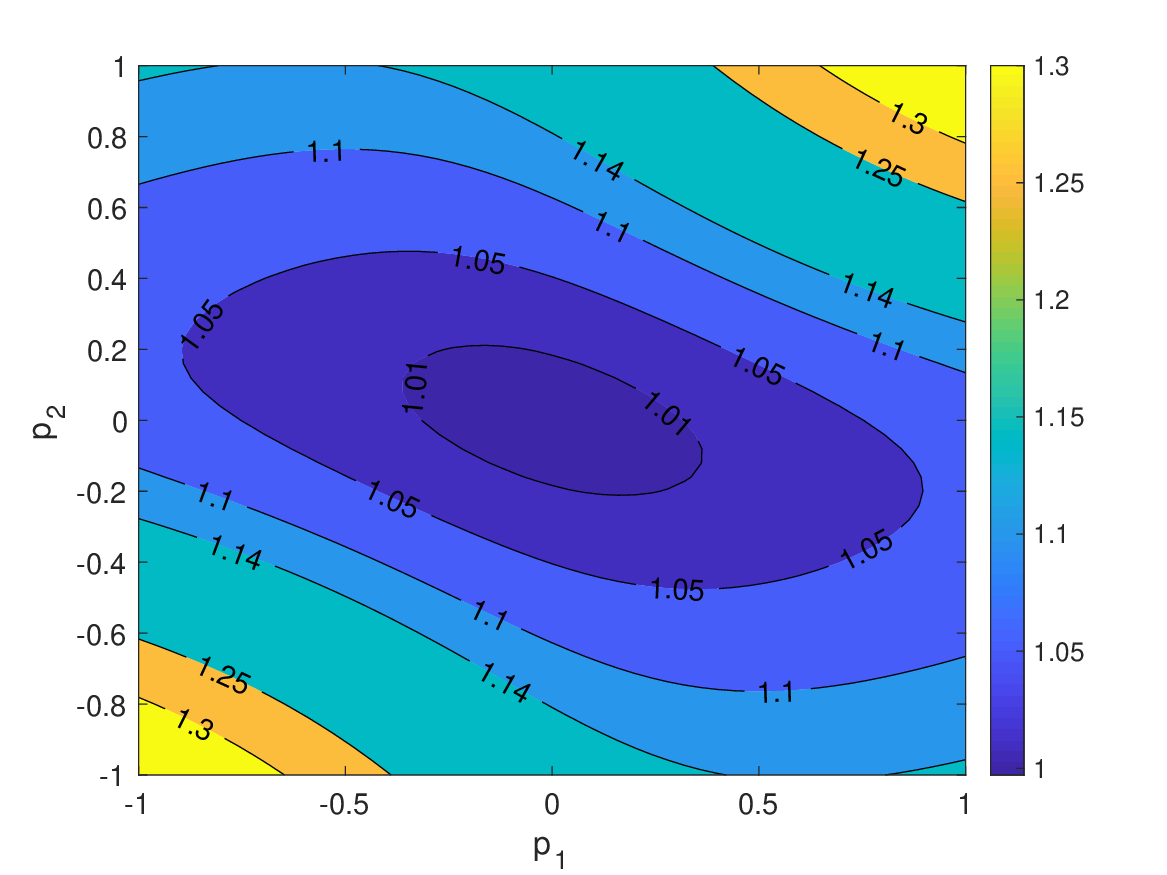}
	\label{fig-sub4a}}
	\subfigure[Level sets of $\mathcal{R}_0$ v.s. $q_1$, $q_2$ with $p_1=p_2=0$]{\includegraphics[width=0.45\textwidth]{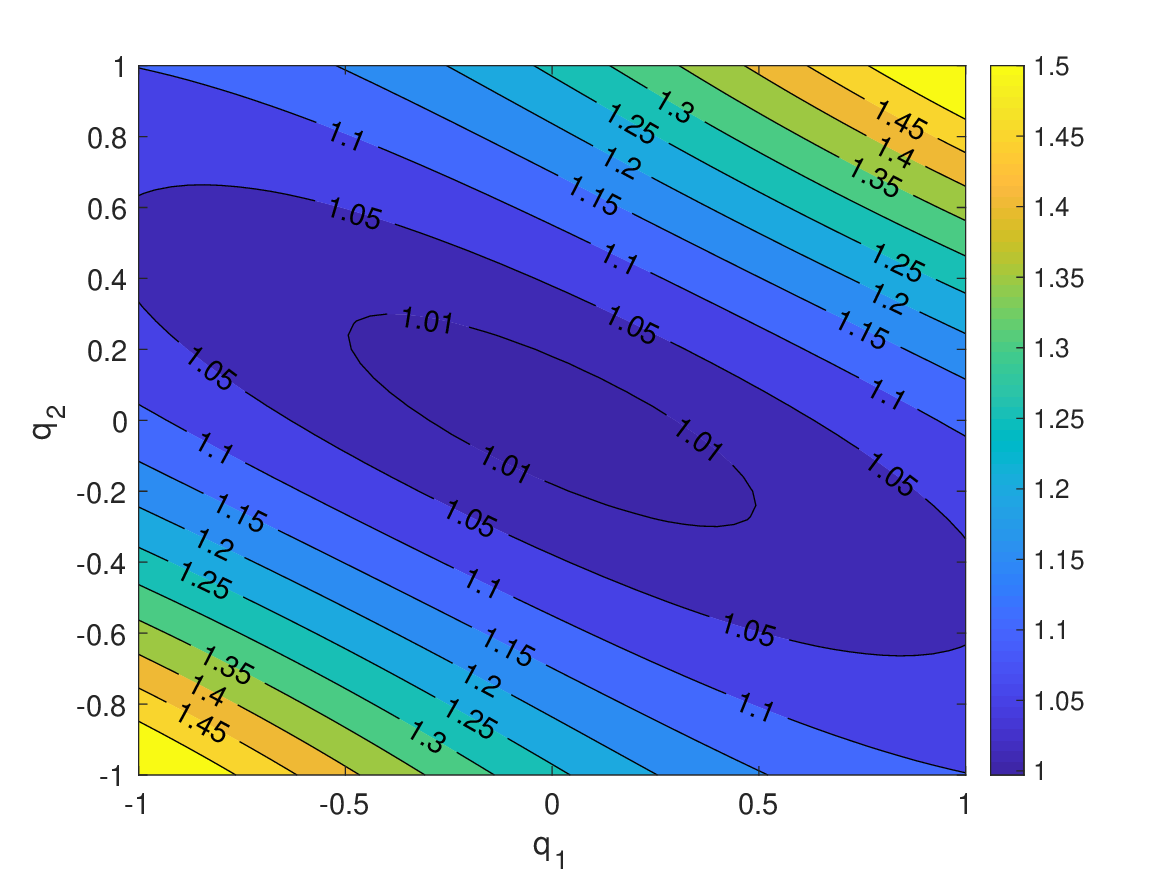}
	\label{fig-sub4b}}\\
		\centering
	\subfigure[Level sets of $\mathcal{R}_0$ v.s. $p_1=q_1$, $p_2=q_2$]
{\includegraphics[width=0.45\textwidth]{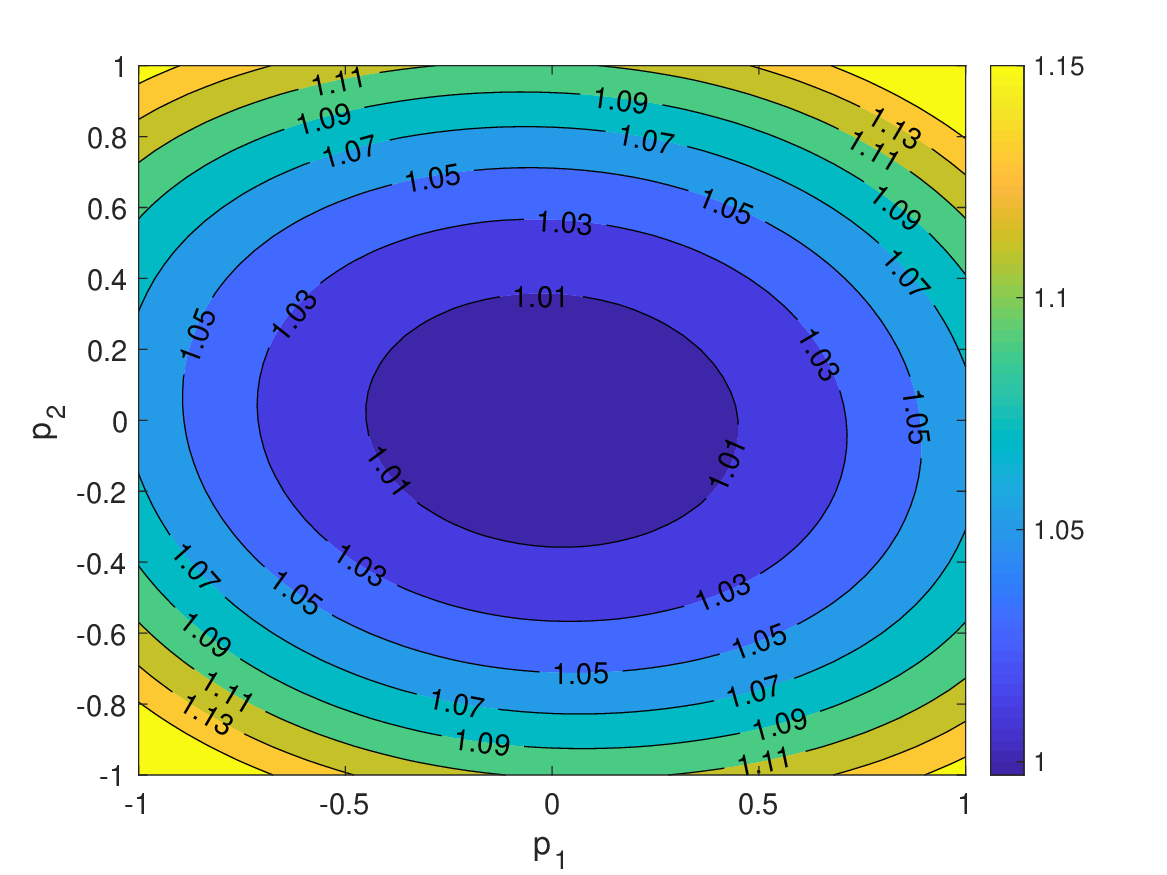}
		\label{fig-sub4c}}
	\subfigure[Level sets of $\mathcal{R}_0$ v.s. $p_1=-q_1$, $p_2=-q_2$]{\includegraphics[width=0.45\textwidth]{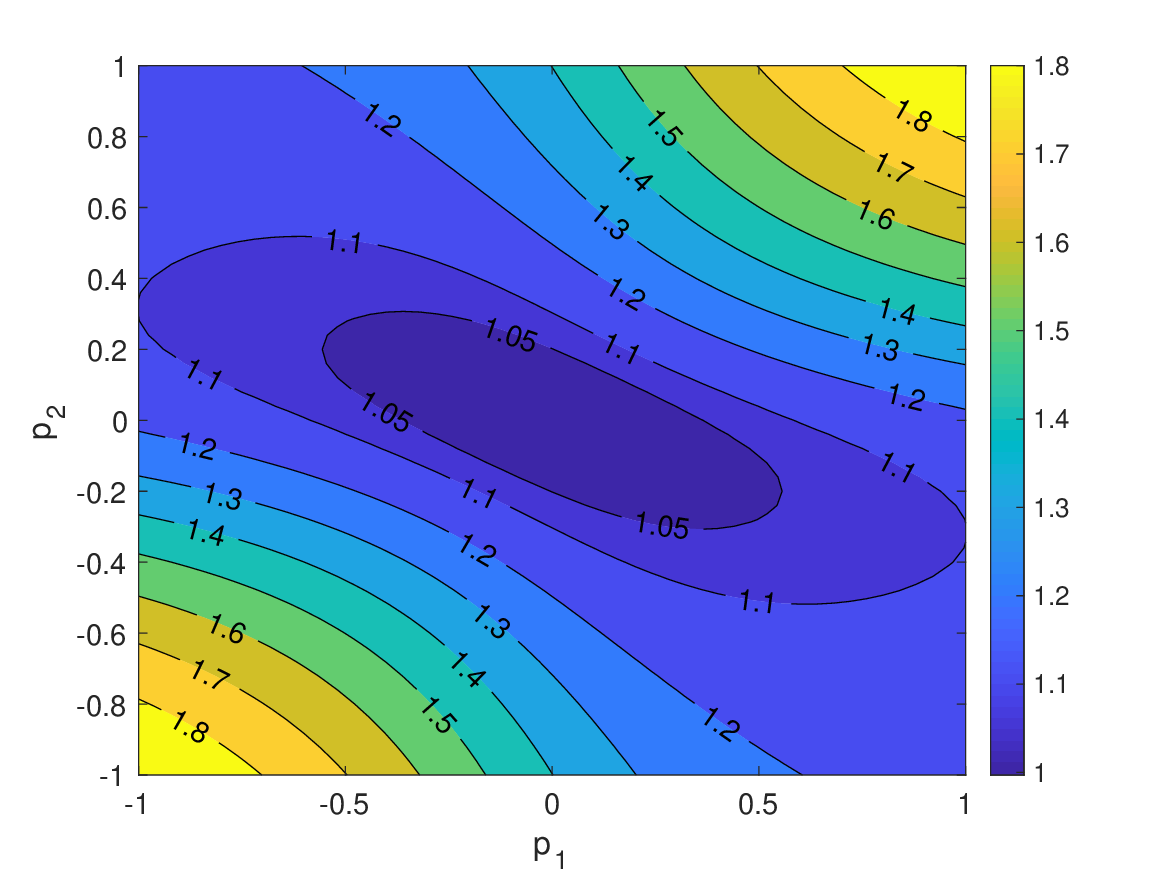}
		\label{fig-sub4d}}\\[-2mm]
	\caption{\small Level sets of $\mathcal{R}_0$ with $p_3=p_4=q_3=q_4=0$.  Simultaneous changes of multiple parameters, compared to changes of individual parameters, can either increase or decrease the impact on $\mathcal{R}_0$.\label{fig:R0:2}}\vspace{-5mm}
	\end{figure}	

\begin{figure}[htb]
	\centering
	\subfigure[Level sets of $\mathcal{R}_0$ v.s. $p_1$, $p_3$ with $p_2=p_4=0$]{\includegraphics[width=0.47\textwidth]{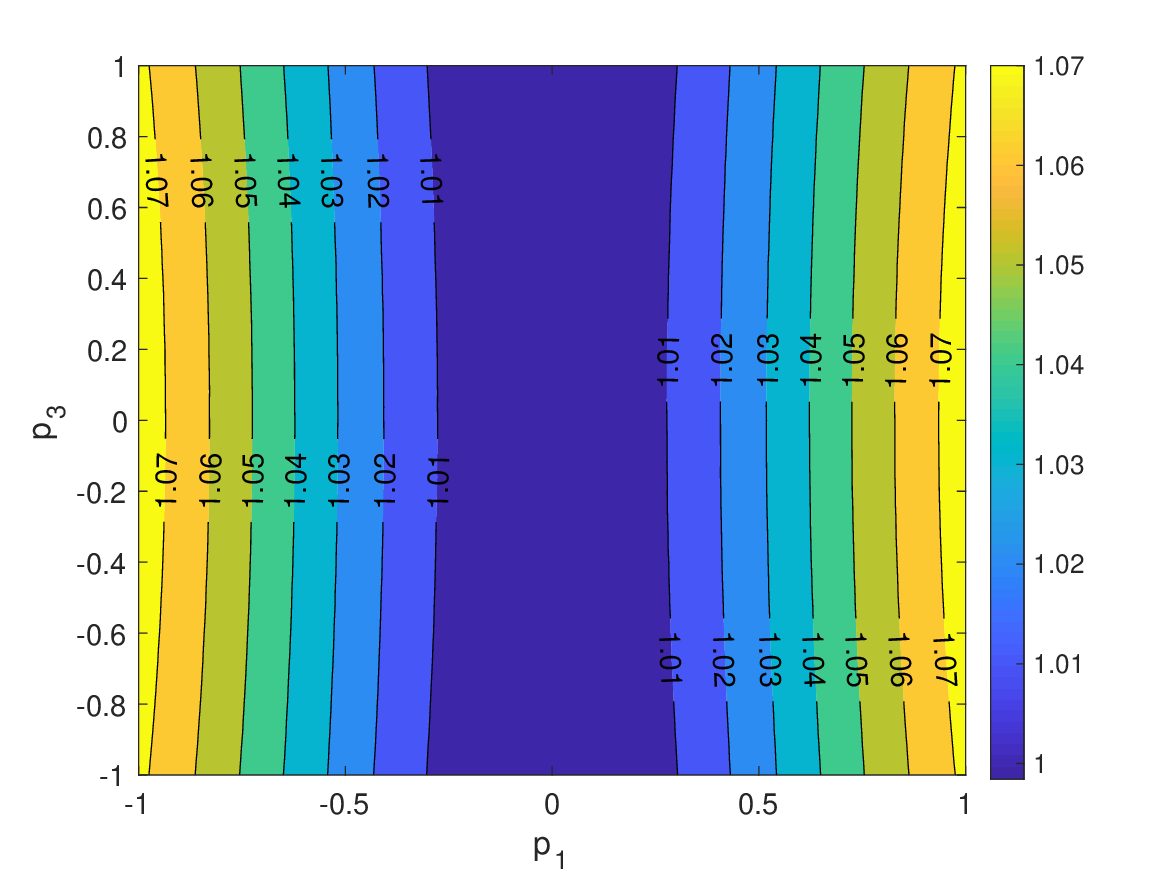}
		\label{fig:sub:13}}	\;
	\subfigure[Level sets of $\mathcal{R}_0$ v.s. $p_2$, $p_4$ with $p_1=p_3=0$]{\includegraphics[width=0.47\textwidth]{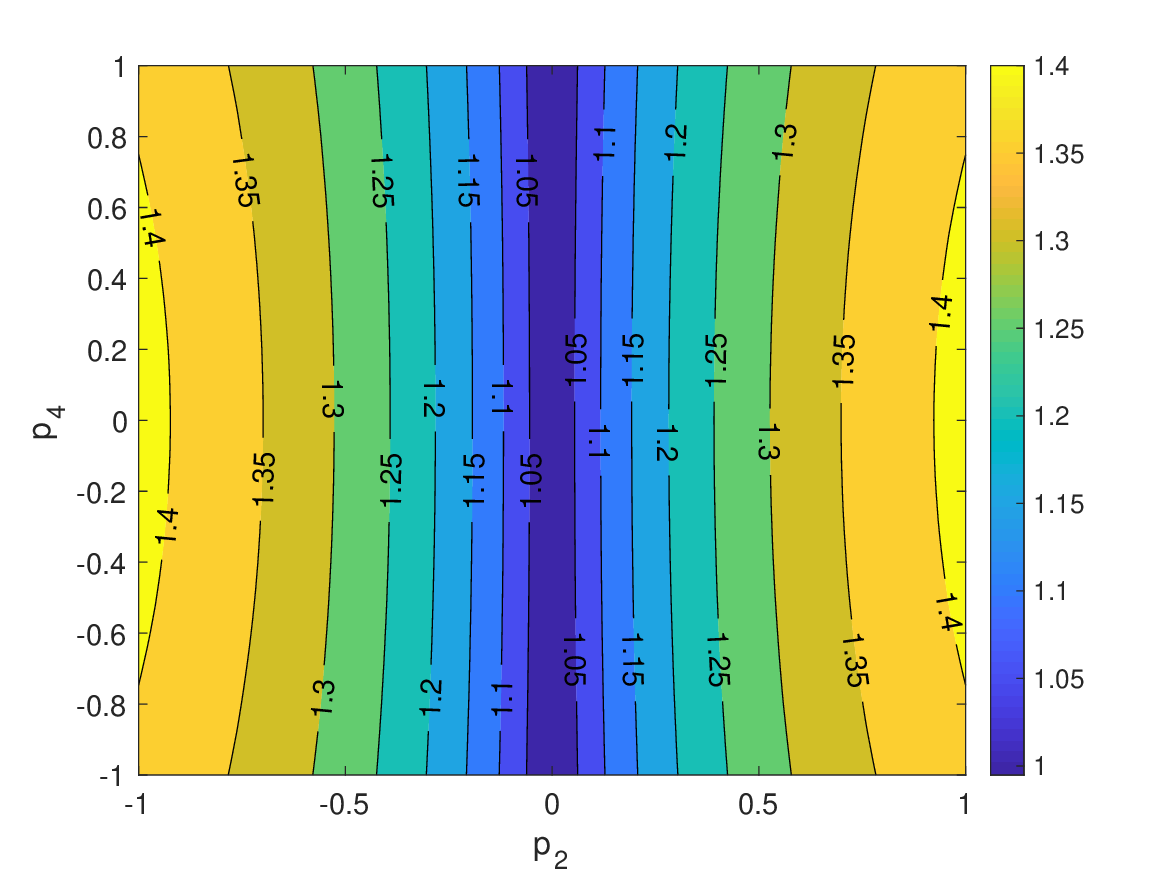}
		\label{fig:sub:24}}	\\[-2mm]
	\caption{\small Level sets of $\mathcal{R}_0$ with $q_1=q_2=q_3=q_4=0$. The spatial heterogeneity may have greater influence on $\mathcal{R}_0$ than the temporal one in general. \label{fig:R0:3}}\vspace{-3mm}
	\end{figure}	

We finish this section with a brief discussion in a biological sense. Recall that $a_1(x,t)$ and $a_2(x,t)$ are the birth rates of susceptible birds and mosquitoes, while $b_1(x,t)$ and $b_2(x,t)$ are the death rates of susceptible birds and mosquitoes. The distribution of sources and the natural environment may yield spatial heterogeneity in these populations, and variations in seasonality may lead to temporal heterogeneity. The spatial and temporal heterogeneity of these parameters first affects the distribution of populations and then influences disease propagation. In general, we can see that spatial heterogeneity increases the risk of disease transmission, while temporal heterogeneity can either increase or decrease this risk.

\section{Discussion}\setcounter{equation}{0} {\setlength\arraycolsep{2pt}

In this paper, we investigate a periodic West Nile virus transmission model \eqref{1.1} that considers the susceptible and infectious populations in birds and mosquitoes, accounting for seasonality and spatial heterogeneity. When humans are considered as hosts, this model can also be used to describe the transmission mechanism of the Zika virus. Using the upper and lower solutions method, we have established the complete dynamic properties of this model. Specifically, we have derived a threshold-type result concerning the existence and uniqueness of positive periodic solution, as well as its global stability in terms of $\mathcal{R}_{0,1}$, $\mathcal{R}_{0,2}$ and $\mathcal{R}_0$:\vspace{-2mm}
 \begin{enumerate}[$(1)$]
\item If $\mathcal{R}_{0,1} \le 1$, then the birds and infected mosquitoes will die out, and so does the disease.
\item If $\mathcal{R}_{0,2} \le 1 $, then the mosquitoes and infected birds will die out, and so does the disease.
\item If $\mathcal{R}_{0,1}, \mathcal{R}_{0,2}>1$ and $\mathcal{R}_0 \le 1$, then the disease will die out.
\item If $\mathcal{R}_{0,1}, \mathcal{R}_{0,2}>1$ and $\mathcal{R}_0>1$,
then the susceptible and infected birds and mosquitoes will eventually stabilize at a positive periodic state.\vspace{-2mm}
\end{enumerate}

We next discuss how to control the disease:\vspace{-2mm}
\begin{enumerate}[$\bullet$]
	\item It is not hard to verify that $\mathcal{R}_{0,j}$  is increasing with respect to $a_j(x,t)$ (birth rate) and decreasing with respect to $b_j(x,t)$ (death rate). Therefore, reducing the birth rates or increasing the mortality rates of birds and viruses can effectively curb the spread of the virus (see the above conclusion (1) and (2)). However, since birds play a crucial role in the ecosystem, such measures should not be implemented against them. The public has long acknowledged that controlling mosquito breeding is the most effective strategy for preventing disease transmission.
	\item Given the significant challenges in managing mosquito populations and the necessity to protect birds without controlling them -- keeping $\mathcal{R}_{0,1}, \mathcal{R}_{0,2}>1$ -- it is crucial to implement measures that decrease $\mathcal{R}_0$ to ensure that it is less than $1$. This purpose can be achieved by reducing $\ell_1(x,t)$ (transmission rates from infected virus uninfected birds) and/or $a_2(x,t)$ (birth rate of susceptible mosquitoes) and/or $\ell_2(x,t)$ (transmission rates   from infected virus to uninfected mosquitos), or enlarging $b_2(x,t)$ (death rate of susceptible mosquitoes). This indicates that vector control strategies play a crucial role in the spread of viruses, such as reducing the number of vector populations and bite rates. Therefore, we can propose some potential control strategies: using insecticides to kill mosquitoes; strengthening the management of vector breeding grounds near human habitats, such as reservoirs, ponds, and water pits, to reduce vector reproduction and growth; implementing personal protective measures, such as screens, door curtains, and mosquito nets, as well as applying mosquito repellents on exposed skin or clothing to avoid contact with mosquitoes and reduce bite rates; and employing transmission-blocking strategies to decrease the likelihood of transmission from infected vectors to hosts with each bite.\vspace{-2mm}
\end{enumerate}

In model \qq{1.1}, effects of recovered hosts is neglected, and terms $H_uV_i$ and $H_iV_u$ represent bilinear incidences. We know that the probability of an infected mosquito biting healthy birds and the probability of a healthy mosquito biting infected birds are $\frac{H_u}{H_u+H_i}$ and $\frac{H_i}{H_u+H_i}$, respectively. It is more reasonable to replace bilinear incidences with standard incidences and consider the effects of recovered hosts (after the virus of the infected host disappears, the infected host becomes a susceptible host again). In such situation, the modified version of \qq{1.1} can be written as
 \bes\left\{\begin{aligned}
&\partial_t H_u=\nabla\cdot d_1(x,t)\nabla H_u+a_1(x,t)(H_u+H_i)-b_1(x,t)H_u\\
&\hspace{17mm}-c_1(x,t)(H_u+H_i)H_u
-\ell_1(x,t)\dd\frac{H_u}{H_u+H_i}V_i+\gamma(x,t)H_i,&&x\in\oo,\; t>0,\\
&\partial_t H_i=\nabla\cdot d_1(x,t)\nabla H_i+\ell_1(x,t)\dd\frac{H_u}{H_u+H_i}V_i-b_1(x,t)H_i\\
&\hspace{17mm}-c_1(x,t)(H_u+H_i)H_i-\gamma(x,t)H_i,\!\!&&x\in\oo,\; t>0,\\
&\partial_t V_u=\nabla\cdot d_2(x,t)\nabla V_u+a_2(x,t)(V_u+V_i)-b_2(x,t)V_u\\
&\hspace{17mm}
-c_2(x,t)(V_u+V_i)V_u-\ell_2(x,t)\dd\frac{H_i}{H_u+H_i}V_u,&&x\in\oo,\; t>0,\\
&\partial_t V_i=\nabla\cdot d_2(x,t)\nabla V_i+\ell_2(x,t)\dd\frac{H_i}{H_u+H_i}V_u-b_2(x,t)V_i\\
&\hspace{17mm}-c_2(x,t)(V_u+V_i)V_i,&&x\in\oo,\; t>0,\\
&\alpha_1\dd\frac{\partial H_u}{\partial\nu}+\beta_1(x,t)H_u=\alpha_1\dd\frac{\partial H_i}{\partial\nu}+\beta_1(x,t)H_i=0,\, &&x\in\partial\oo,\; t>0,\\[1mm]
& \alpha_2\dd\frac{\partial V_u}{\partial\nu}+\beta_2(x,t)V_u=\alpha_2\dd\frac{\partial V_i}{\partial\nu}+\beta_2(x,t)V_i=0,\, &&x\in\partial\oo,\; t>0,\\
&\big(H_u(x, 0), H_i(x, 0), V_u(x, 0), V_i(x, 0)\big)=\big(H_{u0}(x), H_{i0}(x), V_{u0}(x), V_{i0}),&&x\in\oo,\vspace{-2mm}
\end{aligned}\rr.\label{4.10}\ees
where $\gamma(x,t)$ is the recovery rate from the loss of virus in infected hosts to uninfected hosts. When $\alpha_1=\alpha_2\equiv1$ (this is reasonable because $H_u(x,t)+H_i(x,t)\not=0$ in $\bar\Omega\times(0,+\yy)$), it can be shown that the conclusions about problem \qq{1.1} are also true for problem \qq{4.10} except the first and last results in Theorem \ref{th3.1}(3) using the method in the present paper.

\end{document}